\documentclass[final,12pt]{amsart}

\usepackage{amsmath,bm,graphicx,amsthm}
\usepackage{enumerate}
\usepackage{color}
\usepackage{multirow}
\usepackage[centering,text={16cm,20cm},
		marginparwidth=18mm]{geometry}
\usepackage{pdfpages}

\def\Z{\mathbb{Z}}

\def\X{\mathbb{X}}
\def\O{\mathbb{O}}

\def\S{\mathbf{S}}
\def\R{\textbf{R}}
\def\x{\mathbf{x}}
\def\y{\mathbf{y}}
\def\z{\mathbf{z}}

\def\QQ {\mathcal{Q}}
\def\RR {\mathcal{R}}
\def\SS {\mathcal{S}}
\def\FF {\mathcal{F}}

\def\J {\mathcal{J}}
\def\I {\mathcal{I}}

\def\a{\mathfrak{a}}
\def\hh{\mathfrak{h}}
\def\b{\mathbf}
\def\t{\widetilde}
\def\h{\widehat}
\def\A{\alpha}
\def\B{\beta}

\def\{{\lbrace}
\def\}{\rbrace}

\theoremstyle{definition}
\newtheorem{theorem}{Theorem}[section]

\newtheorem{lemma}[theorem]{Lemma}
\newtheorem{proposition}[theorem]{Proposition}

\newtheorem{definition}[theorem]{Definition}

\newtheorem{remark}[theorem]{Remark}
\newtheorem{notation}[theorem]{Notation}

\numberwithin{equation}{section}
\newcommand{\exact}[3]{$0 \longrightarrow #1 \longrightarrow #2 \longrightarrow #3~\longrightarrow~0 $}
\begin{document}
\title
     [Combinatorial knot Floer homology and cyclic branched covers]
     {Combinatorial knot Floer homology\\ and cyclic branched covers}%
\author{Fatemeh Douroudian}%
\address{Department of Mathematics, Faculty of Mathematical Sciences, Tarbiat Modares University, P.O. Box 14115-137, Tehran, Iran.}%
\email{douroudian@modares.ac.ir}%

\author{Iman Setayesh}%
\address{Department of Mathematics, Faculty of Mathematical Sciences, Tarbiat Modares University, P.O. Box 14115-137, Tehran, Iran.}%
\email{setayesh@modares.ac.ir}%
\thanks{}%
\subjclass{}%
\keywords{}%
%\date{}%
%\dedicatory{}%
%\commby{}%
% ----------------------------------------------------------------
\begin{abstract}

Using a Heegaard diagram for the pullback of a knot $K \subset S^3$ in its cyclic branched cover $\Sigma_m(K)$ obtained from a grid diagram for $K$, we give a combinatorial proof for the invariance of the associated combinatorial knot Floer homology over $\Z$.

\end{abstract}
\maketitle
%%%%%%%%%%%%%%%%%%%%%%%%%%%%%%%%%%%%
%%%%%%%%%%%%%%%%%%%%%%%%%%%%%%%%%%%%
% یا نور %
%%%%%%%%%%%%%%%%%%%%%%%%%%%%%%%%%%%%
%%%%%%%%%%%%%%%%%%%%%%%%%%%%%%%%%%%%

\section{Introduction} \label{intro}

Knot Floer homology is an invariant of knots and links in a three-manifold introduced by Ozsv\'ath and Szab\'o \cite{OS} and independently by Rasmussen \cite{Ras}. Computation of this invariant involves counting specific holomorphic disks in symmetric product
of a genus $g$ Heegaard surface. These computations can not be easily done by a computer. In \cite{MOS} Manolescu, Ozsv\'{a}th and Sarkar gave an algorithm that makes these calculations combinatorial. Their methods work for the homology with $\Z_2$-coefficients. 

Parallel to these advances, in \cite{Sarkar} Sarkar and Wang found a combinatorial method to compute the hat version of the Heegaard Floer homology of a three-manifold. Later Ozsv\'ath, Stipsicz and Szab\'o in \cite{nice} gave a combinatorial algorithm to reproduce $\widehat{HF}(Y)$ and gave a topological proof of its invariance properties. The above homology theories are with $\Z_2$-coefficients. 

Ozsv\'ath, Stipsicz and Szab\'o in \cite{Sign} give a general framework for the assignment of signs to bigons and rectangles (more precisely formal flows). They prove that using such sign assignments one can define a combinatorial Heegaard Floer Homology with coefficients in $\Z$ and show that the proof of the invariance of the homology in \cite{nice} lifts to this case. It is not known whether this homology and the Heegaard Floer homology over $\Z$ coincide.

Given a knot $K \subset S^3$ there  is a planer representation called the grid representation. Using such representation in \cite{MOST}, Manolescu, Ozsv{\'a}th, Szab\'o and Thurston gave a combinatorial proof for the invariance of the combinatorial knot Floer homology with coefficients in $\Z$. Levine \cite{L}  constructs a nice Heegaard diagram for the pullback of a knot $K \subset S^3$ in its cyclic branched cover. This construction allows him to compute the Heegaard Floer homology with coefficients in $\Z_2$. 

In this paper, using the sign assignment of \cite{Sign}, we lift the construction of Levine to obtain the combinatorial knot Floer homology of a knot in its cyclic branched cover with $\Z$-coefficients. We give a combinatorial proof of the invariance of this combinatorial knot Floer homology. Our arguments follow closely the proofs of \cite{D} by the first author. We have to modify certain steps to work for the general case of the cyclic branched cover. During our proof we are able to assign a sign to each $4m$-gon in the Heegaard diagram (see Definition~\ref{oct}). This is the main new ingredient in our proof. We prove that this sign assignments satisfies certain natural commutativity properties. This generalization can be used in future works on the combinatorial versions of the Heegaard Floer homology.

The following is the main theorem of this paper. For the definition of the stable combinatorial knot Floer homology see Definition~\ref{stablehomology}.

\begin{theorem}
The stable combinatorial knot Floer homology is an invariant of the knot.
\end{theorem}

\textbf{Plan of the paper:} In Section \ref{prelim} we review the required backgrounds on grid diagrams, Levine's construction and the sign assignments. In section \ref{proofs} we prove that the stable combinatorial knot Floer homology is an invariant of the knot.

\section{Preliminaries}\label{prelim}

To be self-contained in this sections we review the construction of Levine together with the required results about the sign assignments of nice Heegaard diagrams. 

\subsection{Grid diagrams}
A grid diagram $G$ for a knot $K \subset S^3$ is an $n\times n$ planar grid equipped with two sets of markings $\X=\{ X \}_{i=1}^{i=n}$ and $\O=\{ O \}_{i=1}^{i=n}$, such that there is exactly one $X$ marking and one $O$ marking in each column and each row. The markings are placed such that after the following procedure we can retrieve the knot $K$ from the grid diagram $G$. We connect the $X$ and the $O$ in each column with a vertical segment, and we connect the $X$ and the $O$ in each row with a horizontal segment such that the horizontal segments underpass the vertical segments at any intersection. We view the grid diagram as a torus $T^2$ in $S^3$ after identifying the opposite sides of the grid. The horizontal arcs in the grid $G$ become the horizontal circles of the torus, and we denote them by $\A_1,\dots ,\A_n$ and call them alpha curves. We denote the vertical circles by $\B_1,\dots ,\B_n$ and call them the beta curves. The multi-pointed Heegaard diagram $(T^2, \bm\A=\{ \A \}_{i=1}^{i=n} ,\bm\B=\{ \B \}_{i=1}^{i=n},\X ,\O)$ represents $(S^3 , K)$. 

Let $\t K$ be the pullback of a given knot $K\subset S^3$ in the $m$-fold cyclic branched cover of $S^3$ branched along $K$ denoted by $\Sigma_m(K)$. Levine \cite{L} gave a construction for a Heegaard diagram of $(\Sigma_m(K) ,\t K)$ to compute $\h {HFK}( \Sigma_m(K) , \t K)$ over $\Z_2$. We apply the sign assignment introduced by Ozsv\'ath, Stipsicz and Szab\'o in \cite{Sign} for a nice Heegaard diagram, and provide a combinatorial proof of the invariance of the knot Floer homology of $( \Sigma_m(K) , \t K)$ over $\Z$.

\subsection{Levine's Construction}\label{Construction}

In order to be self-contained, we review the construction of a nice Heegaard diagram for $( \Sigma_m(K), \t K)$ given in \cite{L}. Let $G$ be a grid diagram for $K\subset S^3$. By identifying the opposite sides of $G$ we obtain a torus $T$. We may assume that $K$ and $T$ intersect transversely at the markings. The vertical segments of $K$ lie on one side of $T$ whch we call it \emph{the outside}, and the horizontal segments of $K$ lie on \emph{the inside} of $T$. Change $K$ isotopically such that the vertical segments of $K$ lie on $T$, and the rest of $K$ is inside $T$. $K$ has a Siefert surface $F$ that is contained in a ball inside $T$ (for details see \cite{L}). Change $K$ isotopically such that once again it becomes transverse to $T$. Now there are $n$ strips of $F$ outside the torus which intersect the torus in the vertical segments connecting $X$ and $O$ markings in each column. Using this Seifert surface, we form the $m$-fold cyclic branched cover of $S^3$ branched along $K$, i.e. $\Sigma_m(K)$. 
 
 The next step is to give a description of a Heegaard diagram for $\Sigma_m(K)$. We consider $m$ copies of $T$ and denote them by $T_1,\dots,T_{m}$. Let $\t T$ be the surface obtained by gluing $T_1,\dots,T_{m}$ along the branched cuts connecting $X$ and $O$ markings in each column. We glue different copies as follows. Whenever the $X$ marking is above (resp. under) the $O$ marking in a column, the left (resp. right) side of the branched cut that connects $X$ to $O$ in $T_k$ is glued to the right (resp. left) side of the same cut in $T_{k+1}$. Let $\pi : \t T \longrightarrow T$ denote the projection map. The map $\pi$ is an $m$-sheeted branched cover, with $2n$ branched points $X_i$ and $O_i$ for $i=1,\cdots,n$.
Each $\A$- and $\B$- curve on the grid diagram $G$ has $m$ distinct lifts to $\Sigma_m(K)$. 
Denote by $\t {\B^i_j}$ the lift of $\B_j$ in the $i^{th}$ grid
 for $i=1,\dots , m$ and $j=1,\dots , n$. The lift of $\A_j$ which has intersection with $\t {\B_1^i}$ is denoted by $\t \A^i_j$. We illustrate the Heegaard diagram $\t G= (\t T , \t{\bm\A}, \t{\bm\B}, \O, \X )$ by drawing $m$ grids as in Figure \ref{grid}.

\begin{figure}[h]
\centerline{\includegraphics[scale=0.6]{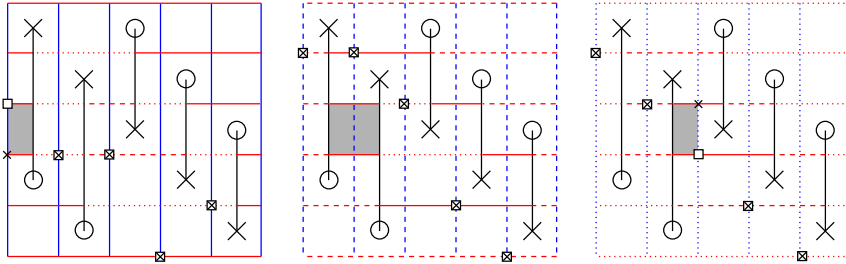}}
\caption {Here we illustrated a Heegaard diagram $\t G=(\t T, \bm{\t \alpha}, \bm{\t \beta}, \O, \X)$ for $\t K \subset \Sigma_2(K)$, where $K$ is the trefoil. The horizontal lines represent different lifts of $\A$-curves which are shown by solid lines, dashed lines and dotted lines.
The vertical lines represent lifts of $\B$-curves. 
Two generators $\x$ and $\y$, and a rectangle in $\pi(\x,\y)$ are shown in the figure. The generator $\x$ is shown with crosses and the generator $\y$ is shown with hollow squares.}
\label{grid}
\end{figure}

Let $\x$ be a subset of intersection points between $\A$- and $\B$-curves in $\t G$ such that each lift of $\A_i$ and each lift of $\B_j$ (for $i,j=1,\dots,n$) has exactly one element in $\x$ (thus $\x$ has $mn$ elements). We call such $\x$ a generator of $\t G$ and denote the set of all generators of $\t G$ by $\S(\t G)$. Each $\x$ can be decomposed (non-uniquely) as the union of the lifts of generators of the grid diagram $G$. That is $\x= \t\x_1\cup \dots \cup \t\x_m$, where $\x_1, \dots, \x_m$ are in $\S(G)$ the set of generators for the grid diagram $G$, see \cite{L}. 

There exists a grading on the set of generators of the Heegaard diagram $\t G$, which is called the Alexander grading. The Alexander grading is defined as follows.  
Given two finite sets of points $A$, $B$ in the plane, let $\J (A, B)$ be the number of pairs
$(a_1,a_2)\in A$ and $(b_1,b_2)\in B$ such that
$(b_1-a_1)(b_2-a_2) > 0$. Given $\x_i \in \S(G)$, define 
\begin{center}
$A(\x_i)=\dfrac{1}{2} \J(\x_i-\dfrac{1}{2}(\X+\O),\X-\O)- (\dfrac{n-1}{2})$.
\end{center}
Here we consider $\J$ as a bilinear function of its two variables. 

Given a decomposition for $\x \in \S(\t G)$ as $\x= \t\x_1\cup \dots \cup \t\x_m$, where $\x_i \in \S(G)$ for $i=1,\dots, m$, we define the Alexander grading of $\x$ to be the average of the Alexander gradings for the generators $\x_i \in \S(G)$ for $i=1,\dots, m$. It can be easily shown that this definition is independent of the chosen decomposition, hence it is well-defined.

A \emph{rectangle} $r$ is a topological disk whose upper and lower edges are arcs of alpha curves, and whose left and right edges are arcs of beta curves. We denote the set of rectangles in $\t G$ by $\RR(\t G)$. 

Given two generators $\x,\y \in \S(\t G)$ that differ in exactly two components, a rectangle $r\in \RR(\t G)$ is said to be from $\x$ to $\y$ if the upper right and lower left corners of $r$ be components of $\x$ and the other two corners of $r$ be components of $\y$. Let $\pi(\x,\y)$ be the set of rectangles from $\x$ to $\y$ that have no $\x$ component in their interiors.

Let $C(\t G)$ be the free abelian group generated by $\S(\t G)$.  We will make  $C(\t G)$ into a chain complex over $\Z$ by defining the boundary operator. In order to define the boundary operator we need to assign signs to rectangles. This is done in sub-section~\ref{comb}.

%%%%%%%%%%%%%%%%%%%%%%%%%%%%%%%%%%%%
%%%%%%%%%%%%%%%%%%%%%%%%%%%%%%%%%%%%
%%%%%%%%%%%%%%%%%%%%%%%%%%%%%%%%%%%%
%%%%%%%%%%%%%%%%%%%%%%%%%%%%%%%%%%%%

\subsection{Sign Assignment} \label{sign}

In order to assign signs to rectangles (and bigons created from isotopies that will appear later), we use the work of Ozsv\'ath, Stipsicz and Szab\'o in \cite{Sign}. The sign assignments are defined not only for rectangles and bigons (afterward called \emph{flows}) in a nice Heegaard diagram, but rather for formal rectangles and formal bigons. To be self-contained we review some aspects of their work.

\begin{definition}
Let $\bm \A$ and $\bm \B$ be two sets, such that $|\bm \A|=|\bm \B|=N$. A pairing of the two sets together with an assignment of $ \pm 1 $ to each intersection is called a \emph{formal generator} of \emph{power} $N$. Let $\sigma \in S_N$ denote the pairing and 
$\epsilon = (\epsilon_1, \dots , \epsilon_N) \in {\{ \pm 1 \} }^N$ be the signs assigned to intersections, we denote the formal generator by $(\epsilon , \sigma)$. 

Consider $2N$ oriented arcs $\A_1,\cdots,\A_N$ and $\B_1,\cdots,\B_N$ in the plane, such that for each $i$ the arc $\A_i$ intersects only $\B_{\sigma(i)}$ at a single point. The orientations of the arcs are in a way that the local intersection number of $\A_i$ and $\B_{\sigma(i)}$ is $\epsilon_i$. We call this configuration, \emph{a graphical representation} of the formal generator $(\epsilon , \sigma)$.

Note that given such oriented arcs we can recover the formal generator $\x$ from the local intersections. Also the classes of graphical representations under the action of the group of orientation preserving homomorphisms of the plane are in bijection with the formal generators. We denote the set of all formal generators of power $N$ by $\S_N$.
\end{definition}

\begin{definition}
Let $\x = (\epsilon , \sigma)$ and $\y = (\epsilon ' , \sigma ')$ be two formal generators of power $N$ satisfying the following properties.
\begin{itemize}
\item There exists $i_0 \in \{ 1, \dots , N \}$ such that $\epsilon_{i_0} = -\epsilon_{i_0} '$ and $\epsilon_j = \epsilon_j '$ for all $j\neq i_0$,
\item $\sigma = \sigma '$.
\end{itemize}
Given $2N$ arcs $\A_1,\cdots,\A_N$ and $\B_1,\cdots,\B_N$ in the plane, such that for each $j\neq i_0$ the arc $\A_j$ intersects only $\B_{\sigma(j)}$ at a single point, and  $\A_{i_0}$ intersects $\B_{\sigma({i_0})}$ in two points. 
We fix an orientation of the arcs such that the local intersection number of $\A_j$ and $\B_{\sigma(j)}$ is $\epsilon_j$ (for $j\neq i_0$). The intersection numbers of $\A_{i_0}$ and $\B_{\sigma({i_0})}$ will be $\epsilon_{i_0}$ and $\epsilon_{i_0} '$. Therefore if we consider a neighborhood of one intersection of $\A_{i_0}$ and $\B_{\sigma({i_0})}$ with the rest of the arcs we will get a graphical representation of $\x$, and using the other intersection point gives a graphical representation of $\y$. We may assume that no other arc intersects the bounded region between $\A_{i_0}$ and $\B_{\sigma({i_0})}$. We require that the induced orientation from the plane to $\A_{i_0}$ as the boundary of this bounded region is from the component of $\x$ to the component of $\y$. We have an action of the group of orientation preserving homomorphisms of the plane on such configurations. 

We call the class of this configuration (under the action), a \emph{formal bigon} of power $N$ from the formal generator $\x = (\epsilon , \sigma)$ to the formal generator $\y = (\epsilon ' , \sigma ')$. We denote such formal bigon by $\mathcal{B} :\x \rightarrow \y$, and call $\x$ (respectively $\y$) the initial (respectively terminal) point of the formal bigon $\mathcal{B}$.
\end{definition}

\begin{definition}\label{formalrec}
A \emph{formal rectangle} of power $N$ from the formal generator $\x = (\epsilon , \sigma)$ to the formal generator $\y = (\epsilon ' , \sigma ')$ (both of power $N$) is the following data.
\begin{itemize}
\item A choices of distinct $s,n \in \{ 1, \dots , N \} $ and distinct $e,w \in \{ 1, \dots , N \} $ such that $\sigma(s)=w$ , $\sigma(n)=e$, 
$\sigma'(s)=e$ , $\sigma '(n)=w$,
$\sigma(j)=\sigma '(j)$ for all $j\neq s, n$
and $\epsilon_j = \epsilon_j '$ for all $j\neq n, s$. 
\item Given $2N$ arcs $\A_1,\cdots,\A_N$ and $\B_1,\cdots,\B_N$ in the plane, such that for each $j\neq s,n$ the arc $\A_j$ intersects only $\B_{\sigma(j)}$ at a single point. Also $\A_{s}$ and $\A_n$ intersect $\B_{w}$ and $\B_e$ to form a rectangle, which we denote by $\R$.

We require an orientation on the arcs such that the local intersection of  the arc $\A_j$ with $\B_{\sigma(j)}$ is $\epsilon_i$ (for $j\neq s,n$ ). 
We also require the existence of an orientation for the four edges of the rectangle $\R$ ($\alpha_s$,$\alpha_n$,$\beta_w$ and $\beta_e$) such that the local intersection signs for the initial pairs of points $SW$ (intersection of $\alpha_s$ and $\beta_w$) and $NE$ coincide with the signs $\epsilon_s$ and $\epsilon_n$ respectively. Also the local intersection signs for the terminal pairs of points $SE$ and $NW$ coincide with $\epsilon'_s$ and $\epsilon'_n$ respectively. Note that this condition is equivalent to $\epsilon_s \epsilon_n \epsilon'_s \epsilon'_n =1$. Similarly we have an action of the group of orientation preserving homomorphisms of the plane on such configurations, and we only consider the class of such configuration.
\end{itemize}

We denote such formal rectangle by $\mathcal{R} :\x \rightarrow \y$, and to emphasis the edges also denote it by the sequence $\A_n \rightarrow \B_w \rightarrow \A_s \rightarrow \B_e$. We call this presentation the \emph{arrow notation}. Note that the data of $\sigma(j)$ and $\epsilon(j)$ for $j \neq n,s$, are not explicitly mentioned in the arrow notation and should be understood from the context. We realize the formal rectangles as embedded rectangles in a plane, and in this way we are able to talk about north, east, south and west edges. If we rotate the rectangle by $180$ degrees it does not change the local intersection numbers, and therefore the class of the configuration remains the same. Hence the arrow notation associated to a given formal rectangle can be shifted and it remains the same, see Figure~\ref{drawing3-1}. More precisely, we have 
$$(\A_n \rightarrow \B_w \rightarrow \A_s \rightarrow \B_e) = (\A_s \rightarrow \B_e \rightarrow \A_n \rightarrow \B_w) .$$
\end{definition}

\begin{figure}[h]
\centerline{\includegraphics[scale=0.6]{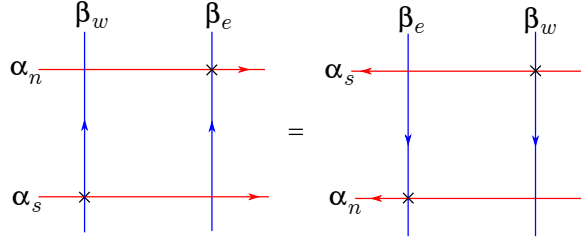}}
\caption{Two equivalent configurations.}
\label{drawing3-1}
\end{figure}

\begin{definition}
We call a formal bigon or a formal rectangle, a \emph{formal flow}. We denote a formal flow $\phi$ from a formal generator
$\x = (\epsilon , \sigma)$ to a formal generator $\y = (\epsilon ' , \sigma ')$ by $\phi :\x \rightarrow \y$.
Denote by $\FF_N$ the set of formal flows of power $N$, i.e. both the initial and the terminal formal generators are of power $N$. Let $\x = (\epsilon , \sigma)$, $\y = (\epsilon ' , \sigma ')$ and $\z = (\epsilon '', \sigma '')$ be formal generators of the same power, $\phi_1 :\x \rightarrow \y$ and $\phi_2 :\y \rightarrow \z$ be two formal flows. We can use these two formal flows consecutively to go from $\x$ to $\z$. We call this process the \emph{composition} of $\phi_1$ and $\phi_2$ and denote it by $(\phi_1 , \phi_2)$. Although we call the pair $(\phi_1 , \phi_2)$ the composition of $\phi_1$ and $\phi_2$, this is not necessarily a formal flow.
\end{definition}

\begin{definition}
A composition of two formal flows $\phi_1 :\x \rightarrow \y$ and $\phi_2 :\y \rightarrow \x$ is called a \emph{boundary degeneration} if 
with some orientations and labeling of the arcs, the pair $(\phi_1 , \phi_2)$ has one of the forms in Figure~\ref{dege}.
The boundary degeneration is of Type $\A$ when the circle(s) in Figure~\ref{dege} is decorated with $\A$, and it is of Type $\B$ if the circle(s) is decorated with $\B$.
\end{definition}

\begin{figure}[h]
\centerline{\includegraphics[scale=0.5]{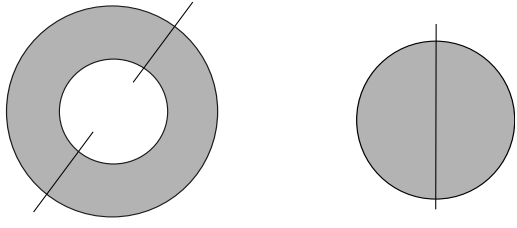}}
\caption {Boundary degenerations of compositions of flows.}
\label{dege}
\end{figure}

Now we are equipped to define a sign assignment of power $N$.

\begin{definition} \label{def sign} 
A sign assignment $\SS$ of power $N$ is a map $\SS :\FF_N \longrightarrow \{\pm 1\}$ that satisfies the following conditions: 

(S-1) For a Type $\A$ boundary degeneration $(\phi_1 , \phi_2)$
\begin{center}
$\SS(\phi_1) \cdot \SS(\phi_2)=1$
\end{center}

(S-2) For a Type $\B$ boundary degeneration $(\phi_1 , \phi_2)$
\begin{center}
$\SS(\phi_1) \cdot \SS(\phi_2)= -1$
\end{center}

(S-3) Given distinct pairs $(\phi_1, \phi_2)$ and $(\phi_3, \phi_4)$ such that the initial formal generators of $\phi_1$ and $\phi_3$ are the same and the terminal formal generators of $\phi_2$ and $\phi_4$ are the same, then
\begin{center}
$\SS(\phi_1) \cdot \SS(\phi_2)+\SS(\phi_3) \cdot\SS(\phi_4)=0$.
\end{center}
\end{definition}

In order to state the main theorem of \cite{Sign}, we need one more definition.

\begin{definition}
Let $\SS$ and $\SS'$ be two sign assignments of power $N$. We say that $\SS$ and $\SS'$ are gauge equivalent sign assignments if there exists a map $u:\S_N \rightarrow \pm 1$ such that for any formal flow $\phi:\x\rightarrow\y$ in $\FF_N$
\begin{center}
$\SS'(\phi)=u(\x)\cdot\SS(\phi)\cdot u(\y)$.
\end{center}
\end{definition}

\begin{theorem}
\cite{Sign} For a given power $N$ there exists a sign assignment, and it is unique up to gauge equivalence.
\end{theorem}

Let $\mathcal{D}=(\Sigma, \bm{\A}, \bm{\B}, \b w, \b z)$ be a nice Heegaard diagram. If $|\A|=N$ we say  that the Heegaard diagram is of \emph{power $N$}. Let $\S$ denote the set of generators of this Heegaard diagram. For $\x,\y \in \S$ we call the set of empty bigons and empty rectangles from $\x$ to $\y$ the set of flows from $\x$ to $\y$, and denote it by $\mathrm{Flows}(\x,\y)$.

Fix an ordering for $\bm{\A}$ and $\bm{\B}$, and an orientation for each $\A$- and $\B$-curves.
This data associates to each generator $\x$ of the Heegaard diagram, a formal generator $\x_f$ of power $N$, and to each flow $\phi$ from $\x$ to $\y$, a formal flow $F(\phi):\x_f \rightarrow \y_f$ of power $N$.
Let $\SS$ be a sign assignment of the same power.

\begin{remark}\label{signremark}
Note that a sign assignment, assigns a $+$ or $-$ sign to each formal flow. For simplicity when the initial and terminal generators are understood from the context, we talk about the sign of the underlying region. The data of those two generators is always present in the definition of the formal flow, even when it is not explicitly mentioned.

Let $\mathcal{D}$ be as above. Given three generators $\x,\y,\z\in\S$, a flow $\phi$ from $\x$ to $\y$ and a disjoint flow $\psi$ from $\y$ to $\z$. Let the underlying region of $\phi$ and $\psi$ be $r$ and $s$. We can use $s$ to go from $\x$ to another generator $\y'$, and by using $r$ from $\y'$ we arrive at $\z$. The (S-3) property implies that 
$$F(r).F(s)=-F(s).F(r).$$
     
The above equality shows that the product of the signs of flows depends on the order that we use them.

\end{remark}

 We define the boundary operator $\t{\partial}^{\Z}(\x)$ for each $\x\in \S$ as follows
$$\t{\partial}^{\Z}(\x)=\sum_{\y\in \S}\sum_{\phi \in \mathrm{Flows}(\x,\y)} \SS(F(\phi))\y .$$

\begin{theorem} \label{signassign}
\cite{Sign} The map $\t{\partial} ^\Z$ satisfies $(\t{\partial}^\Z)^2=0$ (over $\Z$), and the resulting Floer homology $\t{HF}(\mathcal{D};\Z)$ is independent of the choice of $\SS$, the order of the $\A$- and $\B$-curves and the chosen orientation on each of the $\A$- and $\B$-curves. 
\end{theorem}

\subsection{Combinatorial knot Floer homology}\label{comb}

With notations as in sub-section~\ref{Construction}, we will define the boundary map for the complex $C(\t G)$ with coefficients in $\Z$. We fix a sign assignment $\SS$ of power $mn$. We also fix an ordering and orientation for the $\A$- and $\B$-curves.

Given $\x,\y \in S(\t G)$ and a rectangle $r\in \pi(\x,\y)$, we consider the formal rectangle $F(r)$ (with the orientation and the ordering that we fixed on the edges) associate to $r$ and compute $\SS(F(r))$. For $\x\in \S(\t G)$ the boundary operator $\partial$ is defined as follows. 
$$\partial \x= \displaystyle \sum_{\y \in \S(\t G)} \left(\sum_{r\in \pi(\x,\y)} \SS(F(r)) \right) \y .$$

\begin{definition}
Given two pairs $(V,n)$ and $(W,m)$, where $V,W$ are modules of finite rank over $\Z$ and $m,n$ are integers such that $n\geq m$. We say that these pairs are equivalent if $V \cong W \otimes (\Z \oplus \Z)^{n-m}$. If $V,W$ are graded modules the isomorphism is degree preserving, where the copies of $\Z$ are taken to be in degree $0$ and $-1$.
\end{definition}

\begin{definition}\label{stablehomology}
The class of the pair $(H_*(C(\t G)),n)$ is called the stable combinatorial Floer homology of $K$. 
\end{definition}

%%%%%%%%%%%%%%%%%%%%%%%%%%%%%%%%%%%%
%%%%%%%%%%%%%%%%%%%%%%%%%%%%%%%%%%%%
%%%%%%%%%%%%%%%%%%%%%%%%%%%%%%%%%%%%
%%%%%%%%%%%%%%%%%%%%%%%%%%%%%%%%%%%%

\section{Invariance of combinatorial knot floer homology} \label{proofs}

By the work of Cromwell \cite{Cromwell} any two Heegaard diagrams $G$ and $H$ for a knot $K\subset S^3$ can be obtained from each other by a sequence of the following three moves.

\begin{itemize}
\item Cyclic permutation
\item Commutation
\item Stabilization (destabilization).
\end{itemize}

From any diagram of $K$ we obtain a diagram for $\t{K} \subset \Sigma_m(K)$. Therefore any two Heegaard diagrams $\t{G}$ and $\t{H}$ for $\t{K}$ can be obtained by a sequence of the following three moves.

\begin{enumerate}
\item Cyclic permutation: We permute columns (resp. rows) of $G$ which induces a permutation on the set of columns (resp. rows) of each grid of $\t{G}$.

\item Commutation: We consider two adjacent columns (resp. rows) of $G$. If the $X$ and $O$ markings in one of them are between the markings of the other (i.e. the top marking in the first column  is above the top marking of the second column, and the bottom marking of the first column is below the bottom marking of the second column), we switch these columns (resp. rows). This move induces a commutation on the Heegaard diagram $\t{G}$, which in each grid commutes the corresponding adjacent columns (resp. rows). 
\item Stabilization (destabilization): Start with a row (resp. column) of $G$. Add a column between the two markings in this row, and add a row below it. Move one of the markings of the selected row to the new row, and add two markings (an $X$ and an $O$) in the appropriate (uniquely determined) positions in the new column. The induced move on the Heegaard diagram $\t{G}$ adds one row and one column to each grid of $\t{G}$. The markings in each grid come from the markings of $G$.
\end{enumerate}

\subsection{Horizontal/Vertical cuts}

Given a grid diagram $G$ for the knot $K$, in the construction that we discussed in \ref{Construction}, we cut along the vertical lines connecting $X$ and $O$ markings and glue the resulting grids together to obtain a Heegaard diagram for the pullback of $K$ to $\Sigma_m(K)$. We denote this diagram by $\t{G}^v$ and the resulting complex by $(C^v,\partial^v)$. Alternatively we could have done the same with the horizontal lines connecting the markings to obtain $\t{G}^h$ and $(C^h,\partial^h)$.

 In the following subsection we show that both of these construction give the same homology and therefor we can use them interchangeably.

The construction of $\t{G}^v$ is discussed in \ref{Construction}. A few comments about $\t{G}^h$ are in order. To construct $\t{G}^h$ we take $m$ copies of $G$, and cut them along the horizontal lines connecting $X$ and $O$ markings. For each pair of markings in a row, if the $X$ marking is in the left of the $O$ marking, we glue the lower part of $i^{th}$ copy to the upper side of $i+1^{th}$ copy. If the $X$ marking is in the right of the $O$ marking, we glue the upper part of $i^{th}$ copy to the lower side of $i+1^{th}$ copy. 

In this set up since the cuts are horizontal, there are no intersections between the $\A$-curves and the cuts, but $\B$-curves might intersect the cuts, hence they might have segments in different grids . Therefore each lift of any $\A$-curve belongs to exactly one grid of $\t{G}^h$. We denote the lift of $\A_j$ to the $i^{th}$ grid of $\t{G}^h$ by $\t{\A}^i_j$, and $\t\B^i_j$ denotes the lift of $\B_j$ which intersects $\t\A^i_1$.

\begin{lemma}\label{hv-generators}
With notation as above, $\t\A_i^a \cap \t\B_j^b \neq \emptyset$ holds in $\t G^v$ if and only if it holds in $\t G^h$. 
\end{lemma}
\begin{proof}

Let $r$ be the rectangle in $G$ whose upper right corner is the intersection of $\A_i$ and $\B_j$, and its lower left corner is the intersection of $\A_1$ and $\B_1$. Let $X(r)$ (resp. $O(r)$) be the number of $X$ (resp. $O$) markings in $r$.

We start at the intersection of $\t\A_i^a$ and $\t\B_1^a$ in the $a^{th}$ grid of $\t G^v$, and move along $\t\A_i^a$, one column at a time. Let $s$ and $t$ be chosen such that $\t\A_i^a \cap \t\B_u^s \neq \emptyset$ and $\t\A_i^a \cap \t\B_{u+1}^t \neq \emptyset$. In the $u^{th}$ column of $G$, exactly one of the following cases happens.
\begin{itemize}
\item If both $X$ and $O$ marking in that column are in the same side of $\A_j$ (either both are above it or both lie below it), in this case $t=s$.
\item The $X$ marking is above $\A_j$ and the $O$ marking is below it, in this case $t=s+1$. 
\item The $X$ marking is below $\A_j$ and the $O$ marking is above it, in this case $t=s-1$. 
\end{itemize}
Therefore $t-s$ is equal to the number of $O$ markings in the $u^{th}$ column of $r$ minus the number of $X$ markings in it. Since $\t\A_i^a \cap \t\B_1^a \neq \emptyset$, by adding the contribution from each column of $r$ we get
$$\t\A_i^a \cap \t\B_j^b \neq \emptyset \text{ in } \t G^v \Leftrightarrow b-a= O(r)-X(r) .$$

In the horizontal cut configuration $\t G^h$, we start at the intersection of $\t\A_1^b$ and $\t\B_j^b$. We move along the $\t\B_j^b$ one row at a time. Similarly we obtain 
$$\t\A_i^a \cap \t\B_j^b \neq \emptyset \text{ in } \t G^h \Leftrightarrow a-b= X(r)-O(r) .$$
Therefore  
$$\t\A_i^a \cap \t\B_j^b \neq \emptyset \text{ in } \t G^v \Leftrightarrow b-a= O(r)-X(r) \Leftrightarrow \t\A_i^a \cap \t\B_j^b \neq \emptyset \text{ in } \t G^h .$$

\end{proof}

This gives us a natural isomorphism $I_{v,h}$ between the two sets $\S(\t G^v)$ and $\S(\t G^h)$. The isomorphism $I_{v,h}$ is given by sending a generator $\x\in\S(\t G^v)$ to the generator $I_{v,h}(\x) \in\S(\t G^h)$ with the same intersections as $\x$. By abuse of notation we also denote $I_{v,h}(\x)$ by $\x$. By extending $I_{v,h}$ linearly, we obtain an isomorphism between $C^v$ and $C^h$ as groups. In the following lemma we show that this group isomorphism is actually an isomorphism of chain complexes.
\begin{lemma}\label{lemma-hor-ver}
Let $\x,\y \in \S(\t G^v)\cong\S(\t G^h)$. The relation $\y \in \partial\x$ holds in $(C^v,\partial^v)$ if and only if it holds in $(C^h,\partial^h)$.
\end{lemma}

\begin{proof}
We assume that $\y \in \partial\x$ in $(C^v,\partial^v)$. Thus $\x$ and $\y$ differ in exactly two components. Let these components be at the intersections of $\t\A_i^a$ and $\t\A_j^b$ with $\t\B_k^c$ and $\t\B_l^d$ with $i\leq j$ and $k\leq l$. By definition there exist an empty rectangle $\t r^v$ in $\t G^v$ connecting $\x$ and $\y$. 

Let $r_1$ be the rectangle in $G$ whose upper right corner is the intersection of $\A_i$ and $\B_k$, and its lower left corner is the intersection of $\A_1$ and $\B_1$. See Figure~\ref{hor-ver}. Since $\t\A_i^a \cap \t\B_k^c \neq \emptyset$, the argument of Lemma~\ref{hv-generators} shows that $c-a=O(r_1)-X(r_1)$.

\begin{figure}[h]
\centerline{\includegraphics[scale=0.9]{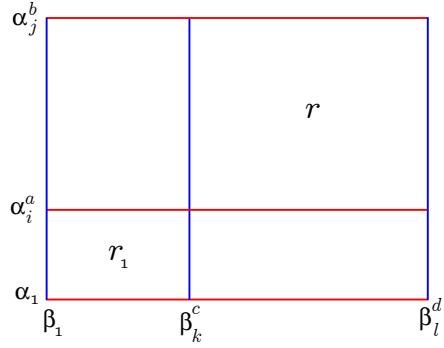}}
\caption {Arrangement of the rectangles $r$ and $r_1$ together with $\A$- and $\B$-curves as in the proof of lemma~\ref{lemma-hor-ver}. }
\label{hor-ver}
\end{figure}

Let us denote the projection of $\t r^v$ to $G$ by $r$. We lift $r$ to $\t G^h$ to obtain $\t r^h$ as follows. By lemma~\ref{hv-generators} we know that in $\t G^h$ the curve $\t\A_j^b$ intersects both $\t\B_k^c$ and $\t\B_l^d$. We lift the upper edge of $r$ to this segment of $\t\A_j^b$. By the lifting property of covering maps fixing the lift of the upper edge of $r$, uniquely determines a lift of $r$ to $\t G^h$. Since we are in the horizontal cut configuration, the lower edge of $\t r^h$ lies on one lift of $\A_i$, which we denote by $\t\A_i^{a'}$. Similar to the vertical case, from $\t\A_i^{a'}\cap\t\B_k^c\neq \emptyset$ we conclude that $c-a'=O(r_1)-X(r_1)$. Thus $a=a'$, and the two rectangles $\t r^v$ and $\t r^h$ have the same boundary curves. Therefore $\t r^h$ is the desired rectangle connecting $\x$ to $\y$ in $\t G^h$.

\end{proof}

The above two lemmas complete the proof of the following lemma.

\begin{lemma}
There is a natural chain map $I_{v,h}:(C^v,\partial^v)\to(C^h,\partial^h)$ which is an isomorphism.
\end{lemma}

\subsection{Cyclic Permutation} \label{perm}

Let $H$ be the grid diagram obtained by cyclically permuting the rows of the grid diagram $G$. The underlying Heegard surface does not change and the set of $\A$- and $\B$- curves are permuted. In the previous subsection we showed that if we use the vertical or horizontal cuts the resulting complexes are isomorphic. 

In order to permute the rows of $G$, we assume that we have horizontal cuts. By cyclically permuting the rows of $G$, the $\A$-curves permute accordingly. Note that the $\A$-curves of $\t G$ and $\t H$ (the Heegaard diagrams of $m$-sheeted branched covers of $G$ and $H$) are related by the same permutation. The $\B$-curves of $\t H$ are obtained by a permutation $\sigma$ of the $\B$-curves of $\t G$ that gives the isomorphism between the associated complexes. The permutation $\sigma$ on $\B$-curves is as follows. If the markings on the first row are in column $i$ and $j$ with $i < j$, the $\B$-curves $\t\B_k^l$ with $k \leq i$ or $k > j$ remain the same. The $\B$-curve $\sigma(\t\B_k^l)$ for $i<k\leq j$ is $\t\B_k^{l+1}$ (resp. $\t\B_k^{l-1}$) if the marking in the $i^{th}$ column is $X$ (resp. $O$). To be more precise, if the intersection of $\t \A_a^b$ and $\t \B_c^d$ on $\t G$ is non-empty, then $\t \A_{a-1}^b \cap \sigma(\t \B_c^d) \neq \emptyset$ on $\t H$. Using the same permutations, we obtain a map between the empty rectangles of $\t G$ and the empty rectangles of $\t H$. The sign assignment to rectangles might change, but in \cite{Sign} they show that if we permute $\A$- and $\B$-curves the homology remains invariant (see Theorem \ref{signassign}). Therefore the knot Floer homology is invariant under the cyclic permutation of rows. 

Similar argument shows that the knot Floer homology is invariant under cyclic permutation of columns.

%%%%%%%%%%%%%%%%%%%%%%%%%%%%%%%%%%%%
%%%%%%%%%%%%%%%%%%%%%%%%%%%%%%%%%%%%
%%%%%%%%%%%%%%%%%%%%%%%%%%%%%%%%%%%%
%%%%%%%%%%%%%%%%%%%%%%%%%%%%%%%%%%%%

\subsection{Commutation} \label{comm}

Let $H$ be the grid diagram obtained by a commutation of two adjacent columns of the grid diagram $G$ for the knot $K$. We call the $\B$-curve between these two adjacent columns  of $G$ the \emph{distinguished} $\B$-curve and denote it by $\B^{dist}$. 

Let $\t H$ be the Heegaard diagram associated to the grid diagram $H$. For a given commutation we define Heegaard diagrams $E_0=\t G, E_1, \dots,E_{m-1} , E_{m}=\t H$ as follows. The Heegaard diagram $E_k$ consists of $m$, $n\times n$ grids, where the first $k$ grids have the same $X$ and $O$ markings as the grid diagram $H$, and the remaining grids have the same markings as $G$. We denote the lift of $\B^{dist}$ in the $(k+1)^{th}$ grid of $E_k$ by $\gamma$, and in the $(k+1)^{th}$ grid of $E_{k+1}$ by $\xi$. Note that $E_k$ and $E_{k+1}$ differ in only one grid i.e. the $(k+1)^{th}$ grid. For simplicity we draw the grid diagrams $E_k$ and $E_{k+1}$ in the same diagram. In this diagram we leave the $m-1$ common grids of $E_k$ and $E_{k+1}$ unchanged, and in the $(k+1)^{th}$ grid we draw both $\gamma$ and $\xi$ as curly lines, such that if we omit the $\xi$ (resp. $\gamma$) we obtain $E_k$ (resp. $E_{k+1}$). This is best understood by looking at Figure \ref{Pentagon}. 

\begin{figure}[h]
\centerline{\includegraphics[scale=0.9]{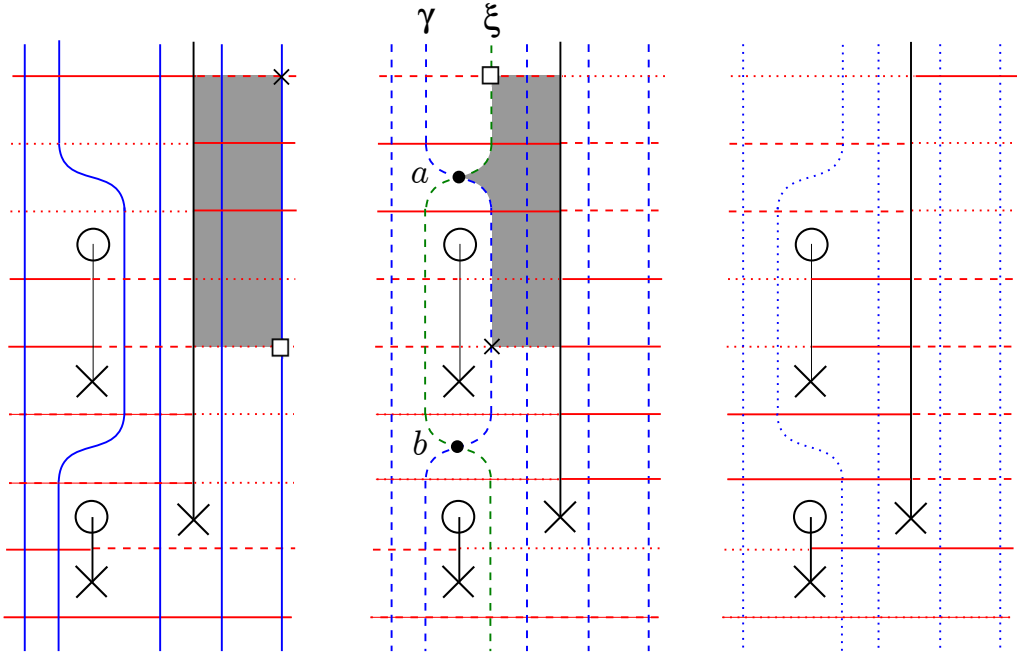}}
\caption {A typical pentagon between two generators $\x$ and $\y$ is shown in gray. The $\x$ components are shown by crosses and the $\y$ components are shown by hollow squares.}
\label{Pentagon}
\end{figure}

Using this method of drawing the grids for $E_k$'s, every grid has the same position of $X$ and $O$ markings.  This allows us to glue the $m$ grids according to the position of markings as before (see subsection~\ref{Construction}).

We denote the $j^{th}$ vertical arc in the $i^{th}$ grid of $E_k$ by $\beta_j^{i}(k)$, where $i=1,\dots ,m$ and $j=1,\dots , n$. 
We denote the lift of $\A_j$ which intersects $\B_1^{i}(k)$ by $\A_j^{i}(k)$.

A generator of the Heegaard diagram $E_k$ consists of $mn$ points of intersections of $\A$- and $\B$-curves, such that each $\A$- (and $\B$-) curve contains exactly one point. The set of all generators of $E_k$ is denoted by $\S (E_k)$.

Note that if we project a generator $\x\in \S (E_k)$ into an $n\times n$ grid, by K\H{o}nig's theorem it can be decomposed as $\x_1\cup\cdots\cup\x_m$, such that each $\x_i$ has exactly one component in each row and each column (also see \cite[Lemma 3.1]{L}). Hence we have $\x= \t\x_1\cup \dots \cup \t\x_m$, in which $\t{\x}_i$ is a lift of $\x_i$. We say that $\x_i$ is \emph{of type} $G$ (resp. $H$) if the component of $\t{\x}_i$ on the lift of $\B^{dist}$ lies in the last $m-k$ grids that have the same markings as $G$ (resp. first $k$ grids that have the same markings as $H$). 

\begin{definition}
With notations as above, we define the Alexander grading of $\x$ to be:

$$A(\x)=\frac{1}{m}\left( \sum_{\t{\x}_i\text{ is of type }G } A_G(\x_i) + \sum_{\t{\x}_i\text{ is of type }H } A_H(\x_i) \right).$$

$A_G$ (resp. $A_H$) is the Alexander grading obtained by considering the markings of $G$ (resp. $H$). Note that this definition is independent of the chosen decomposition of $\x$. If a component of $\x$ is not above $\B^{dist}$ (i.e. a lift of $\B^{dist}$) the contribution of that component to the Alexander grading of $G$ and $H$ are the same. For the components above $\B^{dist}$ the contribution to the Alexander grading depends only on the grid of that component (whether it is in the first $k$ grids or not), and does not depend on the chosen decomposition.

\end{definition}

Here we follow very closely the argument of \cite[Section 3.2]{D} and \cite{MOST}.  To each $E_k$ we associate a chain complex. Let $C(E_k)$ be the free $\Z-$module generated by $\S(E_k)$. In order to define the boundary operator, we need to introduce the set of rectangles $\RR(E_k)$. Each element of $\RR(E_k)$ is a rectangle $r$ which is a topological disk whose upper and lower edges are arcs of $\A$-curves, and whose left and right edges are arcs of $\B$-curves. We assume that the rectangles do not pass through the branched points, so the interior of $r$ has no $X$ or $O$ markings. 

Let $\x, \y $ be two generators in $\S (E_k)$ which agree along all but two components that lie on two vertical circles. We say that a rectangle $r\in \RR(E_k)$ connects $\x$ to $\y$ if it satisfies the following properties.  
\begin{itemize}
\item The lower-left and upper-right corners of $r$ are components of $\x$, and its lower-right and upper-left corners are components of $\y$.
\item $r$ does not contain any components of $\x$ in its interior.
\end{itemize}

We denote the set of rectangles connecting $\x$ to $\y$ by $\pi (\x,\y)$.  We omit the index $k$ from the notation, since it can be recovered from $\x$ and $\y$.

We fix a sign assignment $\SS$ of power $mn$. 
We also fix an orientation and an ordering of $\A$- and $\B$-curves, simultaneously for $\t G$ and $\t H$. This means that the index of the $\B$-curve $\gamma$ in $\t H$ is equal to the index of $\xi$ in $\t G$, and they have the same orientation. The remaining $\A$- and $\B$-curves in $\t G$ are naturally in correspondence with the curves of $\t H$, and the corresponding curves have the same order and orientation. This data induces an orientation and an ordering for the $\A$- and $\B$-curves of each $E_k$.
 
Therefore for $\x,\y \in \S(E_k)$ and any rectangle $r\in\pi(\x,\y)$ we can consider the formal rectangle $F(r)$ (with the orientation and the ordering that we fixed on the edges) associate to $r$ and compute $\SS(F(r))$. Note that we use the same sign assignment $\SS$ for all the $E_k$'s. For $\x\in \S(E_k)$ the boundary operator $\partial_{k}$ is defined as follows. 
$$\partial \x= \displaystyle \sum_{\y \in \S(E_k)} \left(\sum_{r\in \pi(\x,\y)} \SS(F(r)) \right) \y .$$

Note that if $\pi(\x,\y)=\emptyset$ then $\y$ appears with coefficient $0$. Since $\SS$ satisfies property (S-3) of the sign assignment, we have $\partial^2=0$ and we obtain a chain complex.

\begin{lemma}
There exist an anti-chain map $\Phi_k : C(E_k)\to C(E_{k+1})$ that preserves the Alexander grading.
\end{lemma}
\begin{proof}

We denote the lift of $\B^{dist}$ in the $(k+1)^{th}$ grid of $E_k$ by $\gamma$ and in the $(k+1)^{th}$ grid of $E_{k+1}$ by $\xi$. Note that $E_k$ and $E_{k+1}$ differ in only one grid i.e. the $(k+1)^{th}$ grid. As we explained in the beginning of this subsection, we draw the grid diagrams $E_k$ and $E_{k+1}$ in the same diagram. Recall that in this diagram we leave the $m-1$ common grids of $E_k$ and $E_{k+1}$ unchanged, and in the $(k+1)^{th}$ grid we draw both $\gamma$ and $\xi$ as curly lines, such that if we omit the $\xi$ (resp. $\gamma$) we obtain $E_k$ (resp. $E_{k+1}$). See Figure~\ref{Pentagon}.

We denote the upper intersection point of $\gamma$ and $\xi$ by $a$ and their lower intersection point by $b$. Using a small perturbation we assume that neither $a$ or $b$ lies on an $\A$-curve.

Let $\t T$ be the Heegaard surface associated with $\t G$, this surface is obtained by gluing the diagrams of $E_k$ along the branched cuts. Note that if we glue the diagrams of $E_{k+1}$ instead of $E_k$ the resulting topological surface will be the same, and the only difference is the location of the lifts of the $\beta^{dist}$ i.e. $\gamma$ and $\xi$. Therefore we use $\t T$ for both $k$ and $k+1$.  Hence we can consider $\gamma$ and $\xi$ as curves on $\t T$.

We describe the map $\Phi$ on the set of generators of $E_k$. Given a generator $\x\in \S(E_k)$ and a generator $\y \in \S(E_{k+1})$, a pentagon $p$ connecting $\x$ to $\y$ is an embedded disk in $\t T$ such that $\x$ and $\y$ have the same components except at vertices of $p$. The boundary of $p$ consists of five arcs as follows. Starting from the $\x$ component on $\gamma$ and moving counter clockwise on the boundary of $p$ we traverse along an $\A$-circle $\A_i$. We reach the component of $\y$ on $\A_i$, continuing along the $\B$-curve $\B_j$ we reach the component of $\x$ along $\B_j$. Moving along the $\A$-curve $\A_l$ we reach the component of $\y$ on $\xi$. Going through $\xi$ we reach $a$, and continuing along $\gamma$ we arrive at the initial $\x$ component. We require all the angles of $p$ to be acute, and the interior of $p$ to be empty from $X$ and $O$ markings and $\x$ components. See Figure \ref{Pentagon}. We denote the set of pentagons between $\x$ and $\y$ by $Pent_{a}(\x,\y)$. In order to keep track of the boundary arcs, we represent such $p$ by $\A_i \to \B_j \to \A_l \to \xi \to \gamma$ .

Recall that we fixed a sign assignment $\SS$ of power $mn$, together with an orientation and ordering on all the $\A$- and $\B$-curves such that $\gamma$ and $\xi$ have the same orientation and index. This choice implies that the sign of the formal rectangles $\A_i \to \B_j \to \A_l \to \xi$ and $\A_i \to \B_j \to \A_l \to \gamma$ are the same. We define the sign of the pentagon $p=\A_i \to \B_j \to \A_l \to \xi \to \gamma$ to be the sign (resp. minus the sign) of the formal rectangle $\A_i \to \B_j \to \A_l \to \xi$ if $p$ is to the left (resp. right) of $\gamma$ and $\xi$, and denote it by $\epsilon(p)$.

For generator $\x\in \S(E_k)$ we set
$$\Phi_k(\x)=\displaystyle\sum_{\y\in \S(E_{k+1})} \sum_{p\in Pent_{a}(\x,\y)} \epsilon(p) \y .$$

Clearly $\Phi_k$ preserves the Alexander grading, since we considered pentagons with no marking inside them, hence the Alexander grading of such $\x$ and $\y$ are the same. The fact that $\Phi_k$ is an anti-chain map follows readily from the proof of \cite[Lemma 3.2]{D}. See also \cite{MOST}.

\end{proof}

\begin{lemma}
There exist an anti-chain map $\Psi_{k+1}:C(E_{k+1})\to C(E_{k})$ that preserves the Alexander grading.
\end{lemma}
\begin{proof}
Similarly we define $\Psi_{k+1}:C(E_{k+1})\to C(E_{k})$ by considering pentagons with vertex at $b$ connecting a generator $\y\in \S(E_{k+1})$ to a generator $\x\in \S(E_{k})$, which we denote by $Pent_{b}(\y,\x)$. We set
$$\Psi_{k+1}(\y)=\displaystyle\sum_{\x\in \S(E_{k})} \sum_{p\in Pent_{b}(\y,\x)} \epsilon(p) \x .$$
The proof of \cite[Lemma 3.2]{D} shows that $\Psi_{k+1}$ satisfies the required conditions. See also \cite{MOST}.
\end{proof}

\begin{lemma}
For $0\leq k \leq m$ there exist two maps $H_k^{k+1}:C(E_k)\to C(E_k)$ and $H_{k+1}^{k}:C(E_{k+1})\to C(E_{k+1})$ such that the following identities hold.

\begin{eqnarray*}
\mathbb{I}+\Psi_{k+1}\circ\Phi_k+\partial\circ H_k^{k+1} + H_k^{k+1} \circ \partial =0  \\
\mathbb{I}+\Phi_k\circ\Psi_{k+1}+\partial\circ H_{k+1}^{k} + H_{k+1}^{k} \circ \partial =0 
\end{eqnarray*}

\end{lemma}

\begin{proof}

Given $\x,\y\in \S(E_k)$ a hexagon $h$ connecting $\x$ to $\y$ is   an embedded disk in the Heegaard surface $\t T$ such that its boundary is a path from $\x$ to $\y$ which consists of six arcs. If we start from the $\x$ component on $\gamma$, moving counter clockwise along an $\A$-curve $\A_i$, we reach a component of $\y$. Continuing in this way we traverse along $\B_j$ and $\A_l$ to reach a component of $\y$ on $\gamma$. Then we move along $\gamma$ to reach $b$, then along $\xi$ to reach $a$ and finally along $\gamma$ again to reach the starting component of $\x$. See Figure \ref{hexagon}. We represent such $h$ by $\A_i \to \B_j \to \A_l \to \gamma \to \xi \to \gamma$, and denote the set of hexagons connecting $\x$ to $\y$ with no $X$ or $O$ marking and no $\x$ component inside them by $Hex_k^{k+1}(\x,\y)$. Note that such hexagons lie to the left of $\gamma$ and $\xi$. Using the sign assignment $\SS$ of power $mn$ that we fixed to define the sign of pentagons, we define the sign of $h$ to be the sign of the formal rectangle $\A_i \to \B_j \to \A_l \to \gamma$, and denote it by $\epsilon(h)$.

Similarly for $\x',\y' \in \S(E_{k+1})$ we can consider the set of hexagons between $\x'$ and $\y'$ which will lie to the right of $\gamma$ and $\xi$, and we denote the set of all such hexagons by $Hex_{k+1}^k(\x',\y')$. The sign of such hexagon is defined to be the sign of the associated formal rectangle.

Given $\x\in\S(E_k)$ we set $$H_k^{k+1}(\x)=\displaystyle \sum_{\y\in \S(E_k)} \sum_{h\in Hex_k^{k+1}(\x,\y)} \epsilon(h) \y .$$ $H_{k+1}^k$ is defined similarly. The rest of the argument follows readily from the proof of \cite[Proposition 3.2 and 4.24]{MOST}.

\end{proof}

\begin{figure}[h]
\centerline{\includegraphics[scale=0.9]{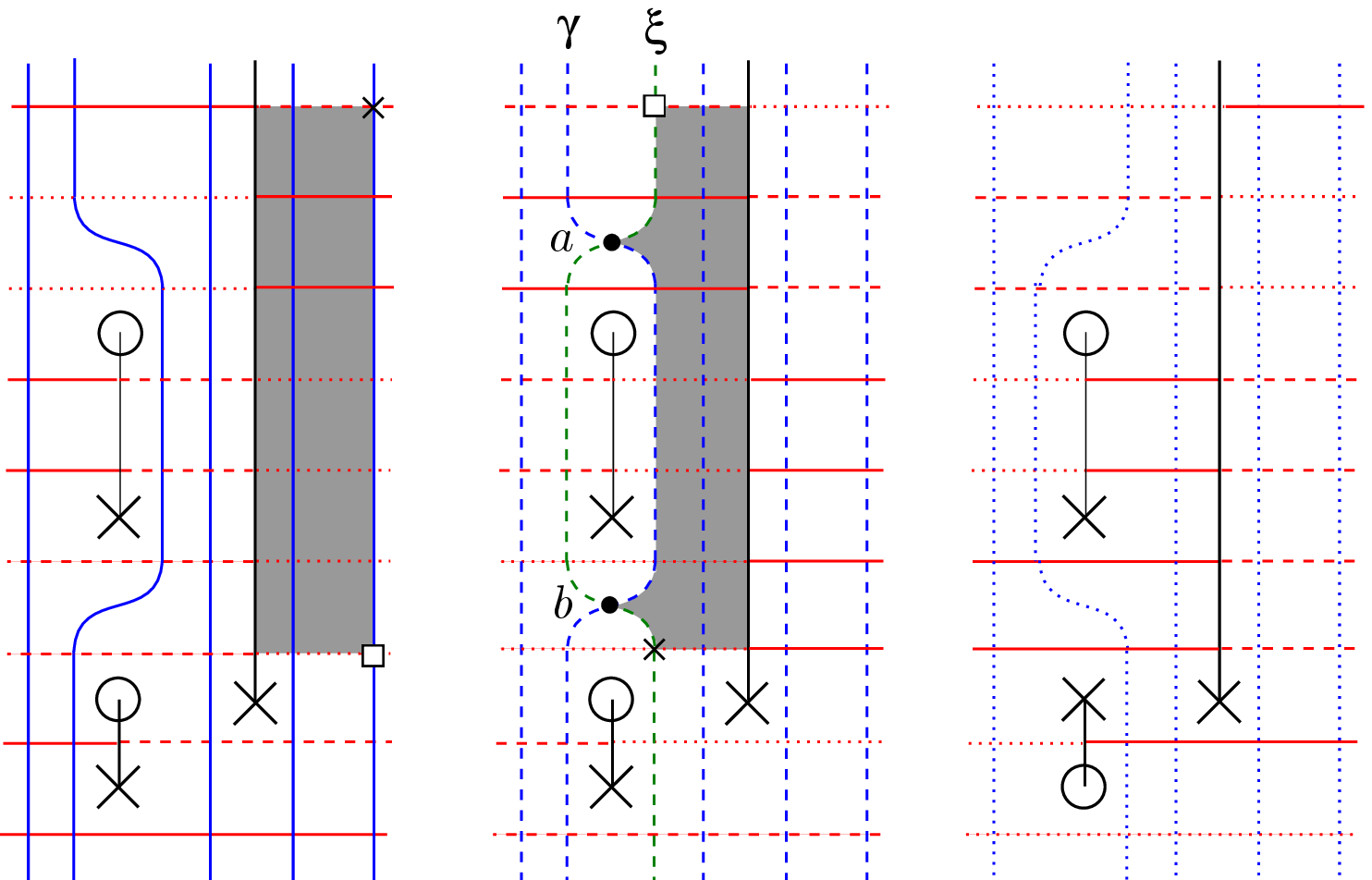}}
\caption {A typical hexagon in shown in gray.}
\label{hexagon}
\end{figure}

The above lemmas show that the chain homotopy between $C(E_k)$ and $C(E_{k+1})$ is a quasi-isomorphism. Composing these maps for $0\leq k \leq m-1$ we obtain a quasi-isomorphism between $C(\t G)$ and $C(\t H)$. Hence the knot Floer homology is invariant under commutation.

%%%%%%%%%%%%%%%%%%%%%%%%%%%%%%%%%%%%%%%%%%%%%%%%%%%%%%%%%%%%%%%%%%%%%%%%
%%%%%%%%%%%%%%%%%%%%%%%%%%%%%%%%%%%%%%%%%%%%%%%%%%%%%%%%%%%%%%%%%%%%%%%%
%%%%%%%%%%%%%%%%%%%%%%%%%%%%%%%%%%%%%%%%%%%%%%%%%%%%%%%%%%%%%%%%%%%%%%%%
%%%%%%%%%%%%%%%%%%%%%%%%%%%%%%%%%%%%%%%%%%%%%%%%%%%%%%%%%%%%%%%%%%%%%%%%
%%%%%%%%%%%%%%%%%%%%%%%%%%%%%%%%%%%%%%%%%%%%%%%%%%%%%%%%%%%%%%%%%%%%%%%%

\subsection{Stabilization} \label{stab}

In this subsection we show that the stable knot Heegaard Floer homology is invariant under stabilization. We assume that the Heegaard diagram $H$ is obtained from $G$ by adding a new column and a new row. Let $X_1$ and $O_1$ be the markings in the new column of $H$. Since the markings $X_1$ and $O_1$ are in consecutive rows, using the commutation move we can assume that $O_1$ is adjacent to the $X$ marking in its row.  We denote the $\A$-curve between $O_1$ and $X_1$ by $\A_1$, the $\B$-curve to the left of $O_1$ is denoted by $\B_1$ and the $\B$-curve to right of $O_1$ is $\B_2$. See Figure~\ref{Stbl}. 

\begin{figure}[h]
\centering
\centerline{\includegraphics[scale=0.8]{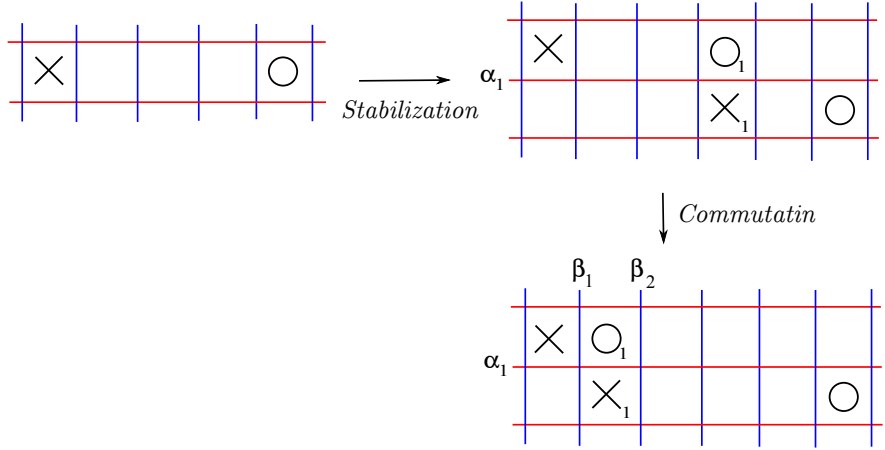}}
\caption{An stabilization followed by commutations moves.}
\label{Stbl}
\end{figure}

During this subsection we introduce the concept of a \emph{pseudo-domain}, which is a generalization of a domain. In order to prove the invariance under the stabilization, we will assign a sign to such pseudo-domains.

\subsubsection{R,L-Shape}

Let $\lambda$ be a curve around the line segment that connects $O_1$ to $X_1$ in the heegaard diagram $H$ of $K$. We denote the inverse image of $\lambda$ in the $k^{th}$ grid of the heegaard diagram $\t{H}$ by $\lambda_k$. 

\begin{definition} Let $\x,\y\in \S(\t H)$. A \emph{pseudo-domain} from $\x$ to $\y$ is a two-chain in the Heegaard diagram $\t H$ whose boundary consists of a path from $\x$ to $\y$ and possibly a number of copies of $\lambda_k$'s. We denote by $\sigma(\x,\y)$ the set of pseudo-domains from $\x$ to $\y$. Note that $\pi(\x,\y) \subset \sigma(\x,\y)$. We define psedo-domains only for combinatorial purposes.
\end{definition}

\begin{remark} In this definition we allow the two-chain to have a number of copies of $\lambda_k$'s in its boundary, but for most parts of the paper we only work with regions with at most one such curve. The reason that we allow multiples of $\lambda_k$'s is to be able to extend the $*$ operator to $\sigma(\x,\y)$.
\end{remark}

There is a natural $*$ operator on domains, given by adding the two-chains. We consider the natural extension of $*$ to pseudo-domains as follows.
Given $\mathbf{a} , \mathbf{b} ,\mathbf{c} \in S(\t H)$, $p \in \sigma(\mathbf{a}, \mathbf{b})$ and $p' \in \sigma (\mathbf{b} , \mathbf{c})$. We can add $p$ and $p'$ as two-chains. By the definition of pseudo-domains, their sum will be an element of $\sigma(\mathbf{a},\mathbf{c})$. In this way we get
$$*:\sigma(\mathbf{a}, \mathbf{b})\times \sigma (\mathbf{b} , \mathbf{c})\longrightarrow \sigma(\mathbf{a},\mathbf{c}) $$

\begin{definition}\label{prect} Let $\x, \y \in S(\t H)$ be two generators that differ exactly in two components located in the $k^{th}$ grid along $\t\B_1^k$ and $\t\B_2^k$. A \emph{punctured rectangle} $\a$ connecting $\x$ to $\y$ is topologically a punctured disk $\t D$ embedded in the Heegaard surface of $\t H$. We require the puncture of $\t D$ to be mapped to the $\lambda_k$. In this case, the boundary of $\a$ consists of $\lambda_k$ and a path from $\x$ to $\y$ consisting of four arcs such that $\partial \a \cap \bm{\t \beta} \subset \t\B_1^k \cup \t\B_2^k$. We denote by $A(\x , \y)$ the set of punctured rectangles from $\x$ to $\y$. Note that the width of a punctured rectangle is exactly $1$ and also $A(\x , \y)$ has at most 1 element. See Figure~\ref{punctured}.

We associate to ${\mathfrak{a}} \in A(\x , \y)$, the formal rectangle $F({\mathfrak{a}}): \x_f \rightarrow \y_f$ in $\FF_{m(n+1)}$ that has the boundary arcs $\partial{\mathfrak{a}}\cap (\bm{\t \A} \cup \bm{\t \B})$. Here $\x_f$ is the formal generator associated with $\x$. We define the sign of the punctured rectangle $\a$ by $\mu(\a) := \SS(F({\a}))$. 
\end{definition}

\begin{remark}
For ${\mathfrak{a}}\in A(\x , \y)$, there is a unique empty rectangle in $Rect(\y, \x)$. We denote it by $r_{\mathfrak{a}}$ and call it the complementary rectangle of $\a$. Note that the support of the union of ${\mathfrak{a}}$ and $r_{\mathfrak{a}}$ is topologically an annulus with a puncture.  From the definition of $\mu(\a)$ and the (S-2) property of the sign assignments we have $\mu(\a) \cdot \SS(F(r_{\a}))=-1$, where $F(r_{\a})$ is the formal rectangle corresponding to $r_{\a}$. 
\end{remark}

\begin{figure}[h]
\centerline{\includegraphics[scale=.5]{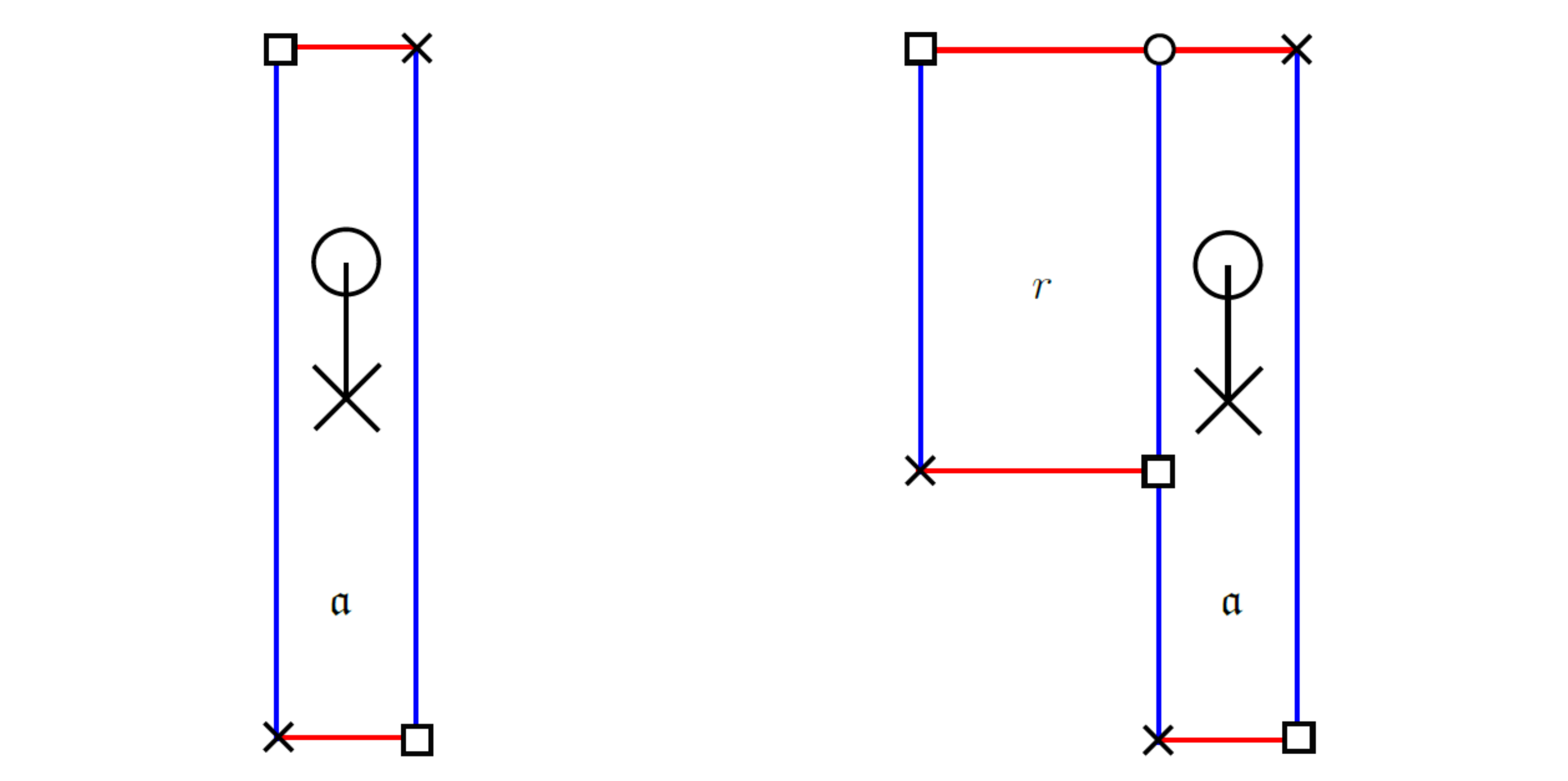}}
\caption{The picture on the left is a punctured rectangle. The picture on the right shows a punctured hexagon.}
\end{figure}

\begin{definition}
Let $\x, \y \in S(\t H)$ be two generators that differ exactly in three components located along $\t\B_1^k$,$\t\B_2^k$ and another $\beta$-curve (say $\t\B_i^j$). A \emph{punctured hexagon} $\hh$ with obtuse corner (or corner for short) at $c\in \t \B_1^k$ connecting $\x$ to $\y$ is the following data.

\begin{itemize}
\item A topologically embedded punctured disk in $\t U$ (the Heegaard surface of $\t H$) with puncture being mapped to $\lambda_k$.
\item If $c$ is in the $k^{th}$ grid (as above), the boundary of $\hh$ consists of $\gamma_k$ and a path from $\x$ to $\y$ consisting of six arcs such that $\partial \hh \cap \bm{\t \beta} \subset \t\B_1^k \cup \t\B_2^k \cup \t\B_i^j$.
\item We require that $\hh$ can be decomposed as the union of a punctured rectangle $\a \in A(\x , \z)$
and $r\in Rect(\z, \y)$, where $\z \in S(\t H)$ differs from $\x$
exactly along $\t \B_1^k$ and $\t \B_2^k$, the upper edge of $r$ and $\a$ are the same $\A$-curve, the lower edge of $r$ contains the point $c$, and the
right (resp. left) edge of $r$ is an arc in $\t \B_1^k$ (resp. $\t\B_i^j$), i.e. $\hh = \a *r$.
\end{itemize}
 We denote by $H(\x , \y)$ the set of punctured hexagons from $\x$ to $\y$. See Figures \ref{punctured} and \ref{LGeneral}.

Note that $c$ is the unique component of $\y$ on $\t \B_1^k$, and the third condition implies that $\hh$ has an obtuse corner at $c$.

\begin{figure}[h]
\centerline{\includegraphics[scale=1]{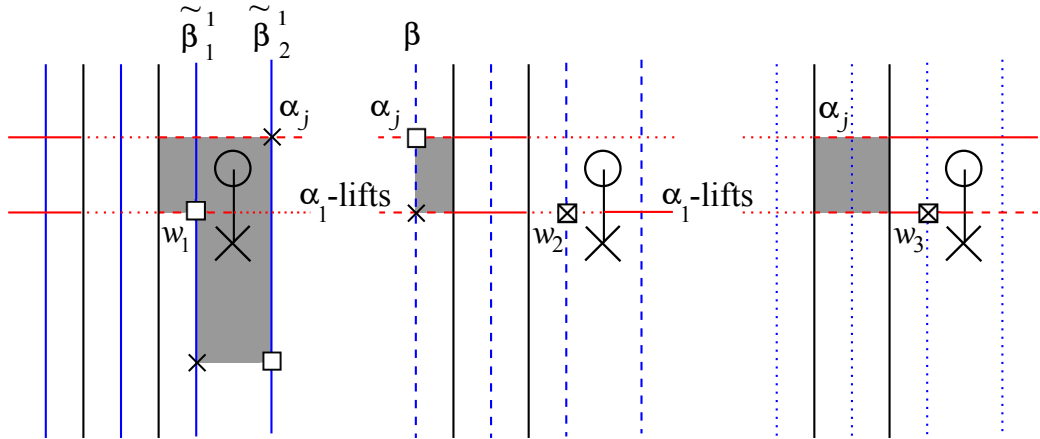}}
\caption {A punctured hexagon from $\x$ to $\y$. The generator $\x$ is shown with crosses and $\y$ is shown with hollow squares. The branched cuts are shown in black.}
\label{LGeneral}
\end{figure}

We define the sign of a punctured hexagon $\hh = \a *r$ as $$\mu(\hh) := \mu(\a) \cdot \SS(F(r)),$$
\noindent where $F(r)$ is the formal rectangle associate with the empty rectangle $r$. 
\end{definition}

\begin{notation}\label{I-index}
We decompose the set of generators of the Heegaard diagram $\t H$ according to the position of the components of a generator on $\t\alpha^i_1$ for $i=1,\cdots,m$. 
We represent the type of each generator with an $m$-tuple with entries $I$, $J$ and $N$. The $i^{th}$ entry is $I$ if the component on $\t\alpha^i_1$ is on one of the lifts of $\B_1$ (i.e. $\t\beta_1^j$ for some $j$). The $i^{th}$ entry is $J$ when the $\t\alpha_1^i$ component is on one of the lifts of $\beta_2$. The $N$ in the $i^{th}$ entry shows that the generator has its $\t\alpha^i_1$ component neither on the lifts of $\beta_1$ nor on the lifts of $\beta_2$. We denote the set of all such $m$-tuples with $\I_m$. Hence we have the following decomposition for the set of generators:
%\[\S(\t H)=(I,I)\cup(I,J)\cup(I,N)\cup(J,I)\cup(N,I)\cup(J,J)\cup(J,N)\cup(N,J)\cup(N,N)\]

$$\S(\t H)= \displaystyle\coprod_{\omega \in \I_m} \omega $$

Having the above decomposition of the set of generators of $C(\t H)$ we get a decomposition of $C(\t H)$ as the direct sum of the sub-modules generated by the generators of the same type.

$$C(\t H)= \displaystyle\bigoplus_{\omega \in \I_m} C^{\omega} $$

 Let $k$ be between $1$ and $m$. We denote by $\I_m^{J,\leq k}$ (resp. $\I_m^{J,k}$) the subset of $\I_m$ consisting of sequences that have at most (resp. exactly) $k$ entries equal to $J$. Let $C^{J,\leq k}$ (resp. $C^{J,k}$) be the complex generated by generators of type $\I_m^{J, k}$ (resp. $\I_m^{J,k}$). Similarly we can define $\I_m^{I,\leq k}$ and $\I_m^{I,k}$ (resp. $\I_m^{N,\leq k}$ and $\I_m^{N,k}$) by replacing the role of $J$ by $I$ (resp. $N$).

\end{notation}

\begin{notation}
We denote the intersection point of $\A_1$ and $\B_1$ in the Heegaard diagram $H$ by $w$. Let $w_i$ be the lift of $w$ to the $i^{th}$ grid of $\t H$. Note that $w_i$ is the intersection of a lift of $\A_1$ with $\t\B_1^i$. 
\end{notation}

\begin{definition} \label{Ldomain} Let $\x \in S(\t H)$ and $\y \in \I_m^{I,m} \subset S(\t H)$. Let $\Omega=\left\lbrace i :  w_i\in\x\right\rbrace$ and $\{ j_1,\cdots,j_k \}:=\{1,\cdots,n \} \setminus \Omega$. A pseudo-domain $p\in \sigma (\x,\y)$ is called of \emph{Type} (L,$k$) (or \emph{Type L} when there is no confusion) if the following condition is satisfied:

\begin{itemize}

\item  For each $l$ with $1\leq l \leq k$ there exists a generator $\z_l \in S(\t H)$ and a punctured hexagon $\hh_l\in H(\z_{l-1},\z_l)$ (we let $\z_0=\x$) with obtuse corner at $w_{j_l}$ such that $p=\displaystyle\bigcup_{l=1}^k \hh_l$. We define the sign of $p$ by $$\displaystyle\mu(p):= \prod_{l=1}^k \mu(\hh_l).$$

\end{itemize}

Note that each $\hh_j$ is the union of two regions (a rectangle and a punctured rectangle), so the above definition is independent of the order of the product. See Remark~\ref{signremark} and Remark~\ref{signremark2} below. If $\Omega=\{1,\cdots,n \}$ then $\x=\y$, and $p$ is the trivial domain. In this case we set $\mu(p):= 1$.  
\end{definition}

\begin{remark}\label{signremark2}
Note that the order of terms in the product $\prod_{l=1}^k \mu(\hh_l)$ is not important since each term $\mu(\hh_l)$ is the product of the signs of two pseudo-domains, and if we switch the order of two punctured hexagons the sign is multiplied by $(-1)^4=1$.
\end{remark}

\begin{definition} \label{oct} For $\x \in S(\t H)$ and $\y \in \I_m^{I,m} \subset S(\t H)$, a \emph{$4m$-gon} $\theta \in \sigma(\x , \y)$ is topologically an embedded disk in $\t U$ (the Heegaard surface of $\t H$) whose boundary is a path form $\x$ to $\y$ that consists of $4m$ arcs, arranged as follows. Starting from the component of $\x$ on $\t\B_1^1$, we traverse the boundary with the orientation that comes from the $4m$-gon. We go along an $\A$-curve to meet a $\y$ component. We traverse a $\B$-curve to reach a component of $\x$ on one of the lifts of $\A_1$, then we go through that lift, passing the branch cut connecting $O_1$ and $X_1$, we meet a component of $\y$ that is on the $\t\B_1^2$. Since $\y \in \I_m^{I,m}$ this component of $\y$ is located at $w_2$. We go through $\t\B_1^2$ and get to a component of $\x$ on $\t\B_1^2$. We continue in this way and go through the components of $\x$ on $\t\B_1^3,\cdots,\t\B_1^{m-1},\t\B_1^m$ and also $w_3,\cdots,w_m,w_1$ and return to the component of $\x$ on $\t\B_1^1$ that we started from. The interior of $\theta$ is empty from the components of $\x$ and basepoints other that $X_1$. See Figure.~\ref{4m-gon}. 

\begin{figure}[h]
\centerline{\includegraphics[scale=1]{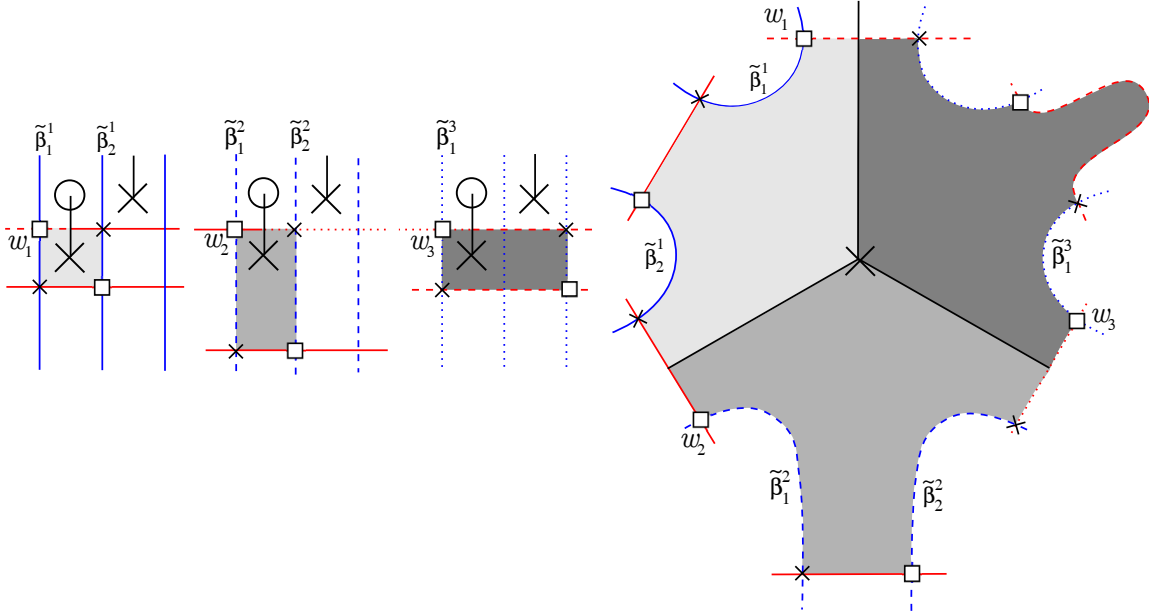}}
\caption {A $12$-gon.}
\label{4m-gon}
\end{figure}

\end{definition}

Given $\x \in S(\t H)$ and $\y \in \I_m^{I,m} \subset S(\t H)$, let $\theta \in \sigma(\x , \y)$ be a $4m$-gon. For simplicity we denote the edges of $\theta$ by $a_1,b_1,a_2,b_2,\cdots,a_{2m},b_{2m}$, in a way that each $a_i$ is an arc of an $\A$-curve and each $b_j$ is an arc of a $\B$-curve. 
We change $a_1$ by a finger move and make it intersect $b_2,b_3,\cdots,b_{2m-1}$. See Figure.~\ref{4m}.

\begin{figure}[h]
\centerline{\includegraphics[scale=1]{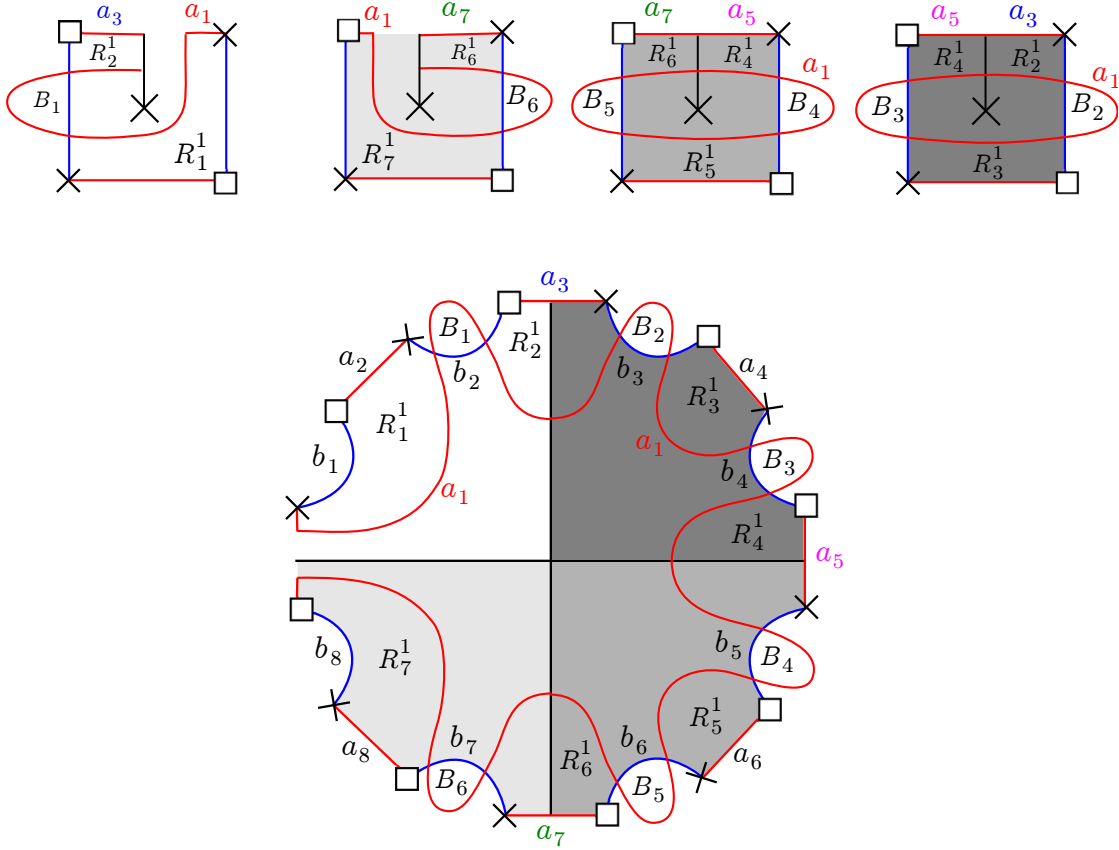}}
\caption {A 4m-gon and the required finger moves to define the sign assignment.}
\label{4m}
\end{figure}

For $i=1,\cdots,2m-1$ we define $R_i^1$ to be the rectangle defined by $a_{i+1} \rightarrow b_i \rightarrow a_1 \rightarrow b_{i+1}$. For $j=1,\cdots,2m-2$ we take $B_j$ to be the bigon defined by $a_1$ and $b_{j+1}$. See Figure.~\ref{4m}. We denote the region obtained from $B_j$ with the reverse orientation by $-B_j$. Note that if we use $R_1^1$ followed by $-B_1$, then $R_2^1$ followed by $-B_2$ and so on till we reach $R_{2m-1}^1$, this composition takes us from $\x$ to $\y$. 

Using the convention of Remark~\ref{signremark} and notations of Figure.~\ref{drawing1}, the S-1 and the S-3 properties of the sign assignment imply that:
$$
\begin{array}{rl}
\SS(F(-B_{2i-1}))\SS(F(R_{2i}^1))\SS(F(-B_{2i}))=&
\SS(CD)\SS(A)\SS(BC)
\\
=&
-\SS(AC)\SS(D)\SS(BC)
\\
=&-\SS(AC)
\\
=&
-\SS(F({R'}_{2i}^1))
\end{array}
$$
\noindent in which ${R'}_{2i}^1$ is the rectangle defined by the same boundary curves $a_{2i+1} \rightarrow b_{2i} \rightarrow a_1 \rightarrow b_{2i+1}$ as ${R}_{2i}^1$, but has different initial and terminal generators. Note that in Figure~\ref{drawing1}, $R_{2i}^1$ (resp. ${R'}_{2i}^1$) has the underlying region $A$ (resp. $AC$).

In order to distinguish between these two formal rectangles we use the following notation. Let $p_i$ be the common vertex of $B_i$ and $R^1_i$, and $q_i$ be the common vertex of $B_i$ and $R^1_{i+1}$. With this notation  $-B_i$ connects a generator $\x_i$ to $\y_i$ that differ in exactly one component. The aforementioned component of $\x_i$ and $\y_i$ are $p_i$ and $q_i$ respectively.

The vertices of ${R}_{2i}^1$ on the curve $a_1$ are $q_{2i-1}$ and $p_{2i}$, and the vertices of ${R'}_{2i}^1$ on $a_1$ are $p_{2i-1}$ and $q_{2i}$. We will say that ${R}_{2i}^1$ is the \emph{actual} rectangle obtained by $a_{2i+1} \rightarrow b_{2i} \rightarrow a_1 \rightarrow b_{2i+1}$ after the finger move, and ${R'}_{2i}^1$ is the \emph{virtual} rectangle obtained from the same boundary curves. The motivation for calling ${R}_{2i}^1$ an actual rectangle comes from Figure~\ref{4m}. This actual rectangle is visualized in this picture, while the virtual one is a formal rectangle that we use in our computations.

\begin{figure}[h]
\centerline{\includegraphics[scale=1]{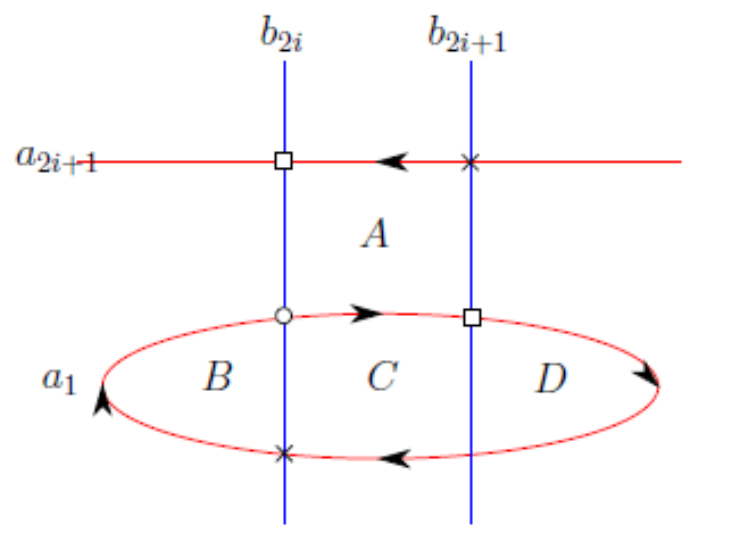}}
\caption{A pictorial representation of composition of formal flows.}
\label{drawing1}
\end{figure}

We define the sign of the $2m$-gon $\theta$ as follows.
\begin{small}
$$\begin{array}{rl}
\mu(\theta,a_1)
&:=\SS(F(R_{1}^{1})) \cdot\SS(F(-B_{1}))\cdot\SS(F(R_{2}^1))\cdot\SS(F(-B_{2}))\cdot\SS(F(R_{3}^{1}))\cdots \SS(F(R_{2m-1}^{1}))
\\
&=(-1)^{m-1} \SS(F(R_{1}^{1}))\cdot \SS(F({R'}_{2}^{1}))\cdot \SS(F(R_{3}^{1}))\cdot \SS(F({R'}_{4}^{1})) \cdots \SS(F(R_{2m-1}^{1})) .
\end{array} $$
\end{small}

The term $a_1$ in $\mu(\theta,a_1)$ indicates that we used the finger move on $a_1$ in order to define the sign of $\theta$. 

Using the same language as above, we can use the finger move on $a_k$ for any $k\in \{1,\cdots,2m\}$. Let $R_{i}^{k}$ be the \emph{actual} rectangle defined by $a_{k+i} \rightarrow b_{k+i-1} \rightarrow a_k \rightarrow b_{k+i}$ (which will appear after the finger move), and ${R'}_{i}^{k}$ be the \emph{virtual} rectangle defined by $a_{k+i} \rightarrow b_{k+i-1} \rightarrow a_k \rightarrow b_{k+i}$ (the indices are considered in a cyclic way modulo $2m$, that is $a_{i}=a_{i+2m}$). 
We define:
$$ \mu(\theta,a_k):=(-1)^{m-1}\SS(F(R_{1}^{k}))\cdot \SS(F({R'}_{2}^{k}))\cdot \SS(F(R_{3}^{k}))\cdot \SS(F({R'}_{4}^{k})) \cdots \SS(F(R_{2m-1}^{k})).$$

We prove that the sign is independent of the edge that we use for the finger move. More precisely we have the following lemma. 

\begin{lemma}\label{indep-lemma}
For any $1\leq i,j \leq 2m$, we have $\mu(\theta,a_i)=\mu(\theta,a_j)$.
\end{lemma}

\begin{proof}
By symmetry it is enough to show that $\mu(\theta,a_1)=\mu(\theta,a_2)$.

During this proof we will use several formal rectangles, that do not neccesairly have a topological representation on the Heegaard diagram.  For example, the diagrams in Figure \ref{drawing2} are not part of the Heegaard diagram. They can be thought of as a visual representation of certains flows, that enable us to use the S-3 property of the sign assignment and obtain the required identities. It should be noted that in any expression regarding the sign of formal flows, the initial and terminal generators must be compatible.
 
Using the S-3 property of the sign assignment for the two diagrams in Figure ~\ref{drawing2}, we obtain the following identities.

%\begin{figure}[h]
%\centering
%\def\svgwidth{360pt} 
%\input{drawing2.pdf_tex}
%\caption{A pictorial representation of composition of formal rectangles.}
%\label{drawing2}
%\end{figure}

\begin{figure}[h]
\centerline{\includegraphics[scale=.4]{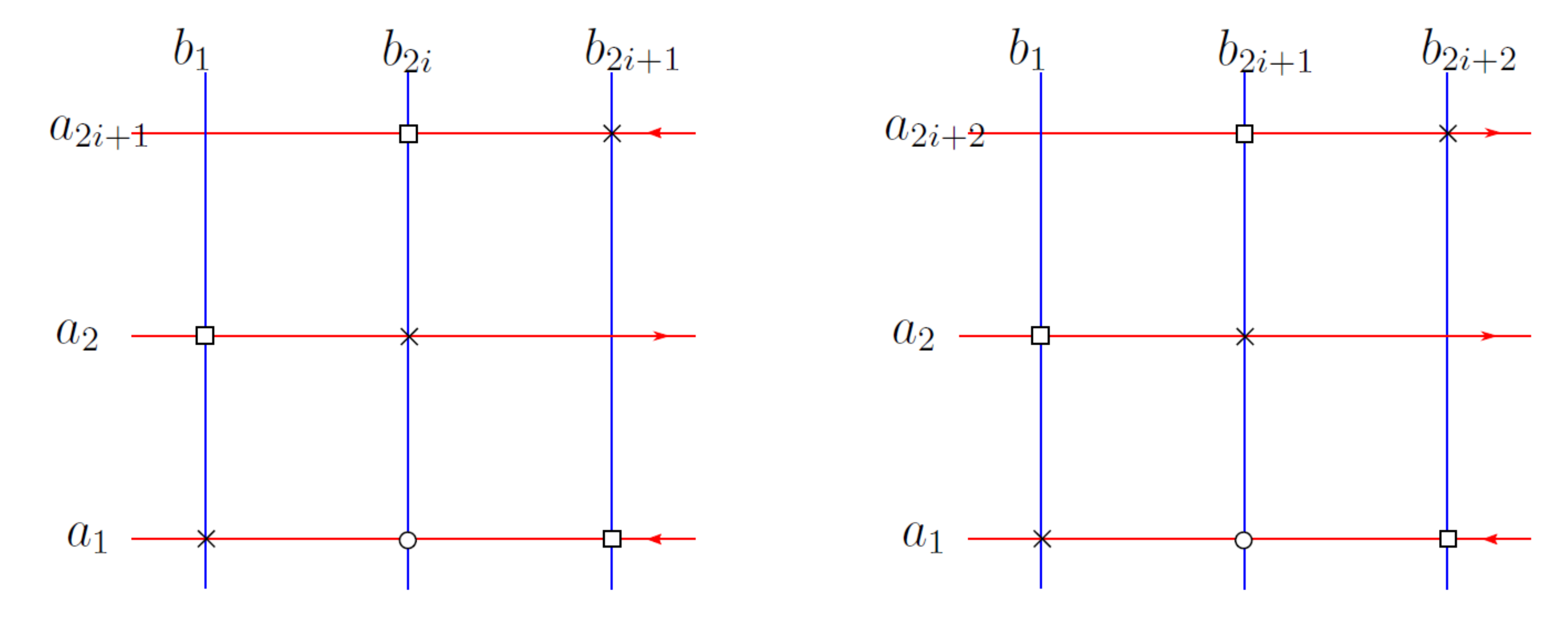}}
\caption{A pictorial representation of composition of formal rectangles.}
\label{drawing2}
\end{figure}

\begin{footnotesize}
$$\begin{array}{l}
\SS(F(a_{2} \rightarrow b_1 \rightarrow a_1 \rightarrow b_{2i}))
\cdot\SS(F({R'}_{2i}^{1}))
\cdot\SS(F({R}_{2i+1}^{1}))
\\
=
\SS(F(a_{2} \rightarrow b_1 \rightarrow a_1 \rightarrow b_{2i}))
\cdot\SS(F(a_{2i+1} \rightarrow b_{2i} \rightarrow a_{1} \rightarrow b_{2i+1}))
\cdot\SS(F(a_{2i+2} \rightarrow b_{2i+1} \rightarrow a_1 \rightarrow b_{2i+2}))
\\ 
=
- \SS(F(a_{2i+1} \rightarrow b_{2i} \rightarrow a_2 \rightarrow b_{2i+1}))\cdot\SS(F(a_{2} \rightarrow b_{1} \rightarrow a_{1} \rightarrow b_{2i+1}))\cdot\SS(F(a_{2i+2} \rightarrow b_{2i+1} \rightarrow a_1 \rightarrow b_{2i+2}))
\\
=
\SS(F(a_{2i+1} \rightarrow b_{2i} \rightarrow a_2 \rightarrow b_{2i+1}))\cdot\SS(F(a_{2i+2} \rightarrow b_{2i+1} \rightarrow a_2 \rightarrow b_{2i+2}))
\cdot\SS(F(a_{2} \rightarrow b_{1} \rightarrow a_{1} \rightarrow b_{2i+2}))
\\
=
\SS(F({R}_{2i-1}^{2}))
\cdot\SS(F({R'}_{2i}^{2}))
\cdot\SS(F(a_{2} \rightarrow b_1 \rightarrow a_1 \rightarrow b_{2i+2}))
\end{array}$$
\end{footnotesize}

Using this equation for $i=1,\cdots,m-1$, allows us to relate $\mu(\theta,a_1)$ and $\mu(\theta,a_2)$.
\begin{small}
$$\begin{array}{rl}
 \mu(\theta,a_1)=
&
(-1)^{m-1} \SS(F(R_{1}^{1}))\cdot \SS(F({R'}_{2}^{1}))\cdot \SS(F(R_{3}^{1}))\cdot \SS(F({R'}_{4}^{1})) \cdots \SS(F(R_{2m-1}^{1})) 
\\
=&
(-1)^{m-1}
\SS(F(R_{1}^{1}))\cdot \displaystyle\left( \prod_{i=1}^{m-1} \SS(F({R'}_{2i}^{1}))\cdot \SS(F(R_{2i+1}^{1})) \right)
\\
=&
(-1)^{m-1}
\SS(F(a_{2} \rightarrow b_1 \rightarrow a_1 \rightarrow b_{2})) \cdot
 \displaystyle \left(\prod_{i=1}^{m-1} \SS(F({R'}_{2i}^{1}))\cdot \SS(F(R_{2i+1}^{1}))\right)
\\
=&
(-1)^{m-1}
\displaystyle \left(\prod_{i=1}^{m-1}
\SS(F(R_{2i-1}^{2}))\cdot \SS(F({R'}_{2i}^{2}))\right)
\cdot  \SS(F(a_{2} \rightarrow b_{1} \rightarrow a_{1} \rightarrow b_{2m}))
\\
=&
(-1)^{m-1}
\displaystyle \left(\prod_{i=1}^{m-1}
\SS(F(R_{2i-1}^{2}))\cdot \SS(F({R'}_{2i}^{2}))\right)
\cdot  \SS(F(a_{1} \rightarrow b_{2m} \rightarrow a_{2} \rightarrow b_{1}))
\\
=&
(-1)^{m-1}
\left(
\displaystyle \prod_{i=1}^{m-1}
\SS(F(R_{2i-1}^{2}))\cdot \SS(F({R'}_{2i}^{2}))
\right) \cdot \SS(F(R_{2m-1}^{2}))
\\
=&
(-1)^{m-1}
\SS(F(R_{1}^{2}))\cdot \SS(F({R'}_{2}^{2}))\cdot \SS(F(R_{3}^{2}))\cdot \SS(F({R'}_{4}^{2})) \cdots \SS(F(R_{2m-1}^{2})) 
\\
=&  \mu(\theta,a_2)
\end{array}
$$

\end{small}

Note that we used the fact that $(a_{2} \rightarrow b_{1} \rightarrow a_{1} \rightarrow b_{2m})=(a_{1} \rightarrow b_{2m} \rightarrow a_{2} \rightarrow b_{1})$ (See Figure~\ref{drawing3}).

%\begin{figure}[h]
%\centering
%\def\svgwidth{400pt} 
%\input{drawing3.pdf_tex}
%\caption{A pictorial representation of two equivalent configurations.}
%\label{drawing3}
%\end{figure}

\begin{figure}[h]
\centerline{\includegraphics[scale=.4]{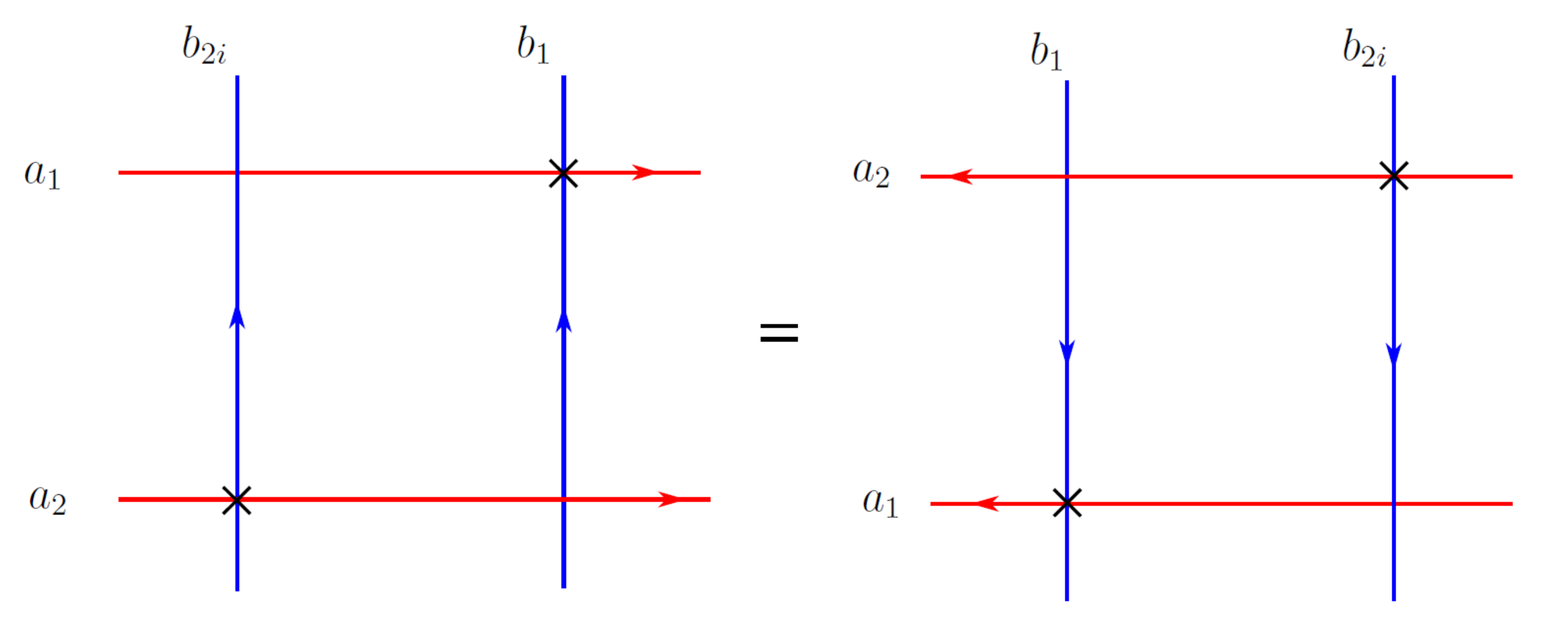}}
\caption{A pictorial representation of two equivalent configurations.}
\label{drawing3}
\end{figure}

\end{proof}

%%%%%%%%%%%%%%%%%%%%%%%%%%%%%%%%%%%%%%
%%%%%%%%%%%%%%%%%%%%%%%%%%%%%%%%%%%%%%

%%%%%%%%%%%%%%%%%%%%%%%%%%%%%%%%%%%%%%
%%%%%%%%%%%%%%%%%%%%%%%%%%%%%%%%%%%%%%
%%%%%%%%%%%%%%%%%%%%%%%%%%%%%%%%%%%%%%
%%%%%%%%%%%%%%%%%%%%%%%%%%%%%%%%%%%%%%
%%%%%%%%%%%%%%%%%%%%%%%%%%%%%%%%%%%%%%
%%%%%%%%%%%%%%%%%%%%%%%%%%%%%%%%%%%%%%

\begin{definition} \label{Rdomain} Let $\x \in S(\t H)$ and $\y \in \I_m^{I,m} \subset S(\t H)$. A pseudo-domain $p\in \sigma(\x,\y)$ is called of \emph{Type R} if the following conditions are satisfied:

\begin{itemize}
\item There exist a set of generators $\{\z_0=\x,\z_1,\cdots,\z_k\}$ such that for each $j=1,\cdots,k$, there exists a punctured hexagon $\hh_j \in H(\z_{j-1},\z_j)$ with corner at $c_j\neq w_{i_j}$ on $\t\B_1^{i_j}$.

\item There exist a $4m$-gon $\theta \in \sigma(\z_k,\y)$.

\item $p=\displaystyle\bigcup_{j=1}^k \hh_j \cup \theta$.
\end{itemize}

We define $$\displaystyle\mu(p):= \prod_{l=1}^k \mu(\hh_l).\mu(\theta).$$

\end{definition}

\begin{definition}\label{F-def}
Let $\x \in S(\t H)$ and $\y \in \I_m^{I,m} \subset S(\t H)$. We denote the set of pseudo-domains $p\in \sigma(\x,\y)$ of Type $L$ by $\sigma^L(\x,\y)$, and the set of Type $R$ pseudo-domains by $\sigma^R(\x,\y)$. We denote by $\sigma^F$ the union of all Type $L$ or $R$ pseudo-domains.

We define a map $F':C(\t H) \to C^{I,m}_Q[1]\oplus C^{I,m}_Q$ as follows.
$$F'(\x)=(\sum_{\y\in\I_m^{I,m}}\sum_{p\in \sigma^R(\x,\y)} \mu(p)\cdot \y , \sum_{\y\in\I_m^{I,m}}\sum_{p\in \sigma^L(\x,\y)} \mu(p)\cdot \y).$$ 
By identifying $C^{I,m}_Q$ and $C(\t G)$, we obtain a map $F:C(\t H) \to C(\t G)[1]\oplus C(\t G)$.

We endow $C(\t G)[1]\oplus C(\t G)$ with the differential
$\partial(a,b)=(-\partial a,\partial b)$
where $\partial$ denotes the differential within $C(\t G)$.
 
\end{definition}

\begin{remark}\label{psi-remark}
We will say a few words about the correspondence between $C^{I,m}_Q$ and $C(\t G)$. Given a generator $\x \in \I_m^{I,m}$, from the definition we know that for each $i\in \{1,\cdots,m\}$, $w_i$ is a component of $\x$. Since $\t G$ is obtained from $\t H$ by removing the $\A$- and $\B$-curves through all the $w_i$'s, if we forget the $w_i$ components of $\x$ we obtain a canonical element in $\S(\t G)$. We denote this map by $\psi:C^{I,m}_Q \to C(\t G)$.  

\end{remark}

\begin{remark}
In order to show that the stable knot Heegaard Floer homology is invariant under the stabilization, we will show that the above $F$ is a quasi-isomorphism. The proof consists of two main pieces. First in Lemma~\ref{part1} we show that $F$ is a chain map. Using a new filtration in Proposition~\ref{quasi-final} we prove that $F$ is a quasi-isomorphism.
\end{remark}

\begin{lemma}\label{part1}
$F$ is a chain map (over $\Z$) and it preserves the Alexander grading.
\end{lemma}

\begin{remark}\label{psi-inverse}
In the course of the following proof we will need to compare $\partial_{\t G} \circ F$ and $F\circ\partial_{\t H}$. Given an empty rectangle $r\in Rect(\x,\y)$ in the Heegaard diagram $\t G$, we can take its inverse image under the identification of Remark~\ref{psi-remark} to obtain a rectangle $r'$ in the Heegaard diagram $\t H$. To be more precise we take the inverse of $\x$ and $\y$ under $\psi$, and we take $r'$ to be the unique rectangle (not necessarily empty) connecting them. The rectangle $r'$ may not be an empty rectangle. If $r'$ is non-empty the only possibility for $r'$ is to contain all the $w_i$'s and both the $O_1$ and $X_1$ markings. We will need to consider these rectangles during the proof. To distinguish them from empty rectangles in the Heegaard diagram $\t H$, we call an empty rectangle (in $\t H$) a Type 1 rectangle, and a non-empty rectangle $r'$ (as above) is called a Type 2 rectangle.
\end{remark}

\begin{remark}\label{type2-sign}
The sign assignment for the Heegaard diagram $\t G$ is obtained from the fixed sign assignment of $\t H$ as follows. Given an empty rectangle $r\in Rect(\x,\y)$ in the Heegaard diagram $\t G$, we consider its inverse under $\psi$ (see Remark~\ref{psi-inverse}) to obtain a rectangle $r'$ in the Heegaard diagram $\t H$. This rectangle as a formal rectangle in $\t H$ has a  sign which we take to be the sign of $r$. 

Note that in \cite{Sign} (during the proof of Proposition~4.4) it has been shown that for a non-empty rectangle $r'$ (in our language a Type 2 rectangle) given by Figure~\ref{type2}  we have the following identity.
$$\SS(F(r'))=\SS(F(r_1))\cdot\mu(\a) \cdot\SS(F(r_2)).$$

More generally in \cite{Sign} it has been shown that for a region as in Figure~\ref{type2}, we have the following identities.
\begin{eqnarray*}
\SS(F(ABCD))&=\SS(B)\cdot\SS(AC)\cdot\SS(D)\\
&=\SS(C)\cdot\SS(BD)\cdot\SS(A).
\end{eqnarray*}

\end{remark}

%\begin{figure}[h]
%\centering
%\def\svgwidth{300pt} 
%\input{type2.pdf_tex}
%\caption{The picture on the left illustrates a Type 2 rectangle. The picture on the right illustrates a rectangle containing a common component of both of the initial and the terminal generators.}
%\label{type2}
%\end{figure}

\begin{figure}[h]
\centerline{\includegraphics[scale=.4]{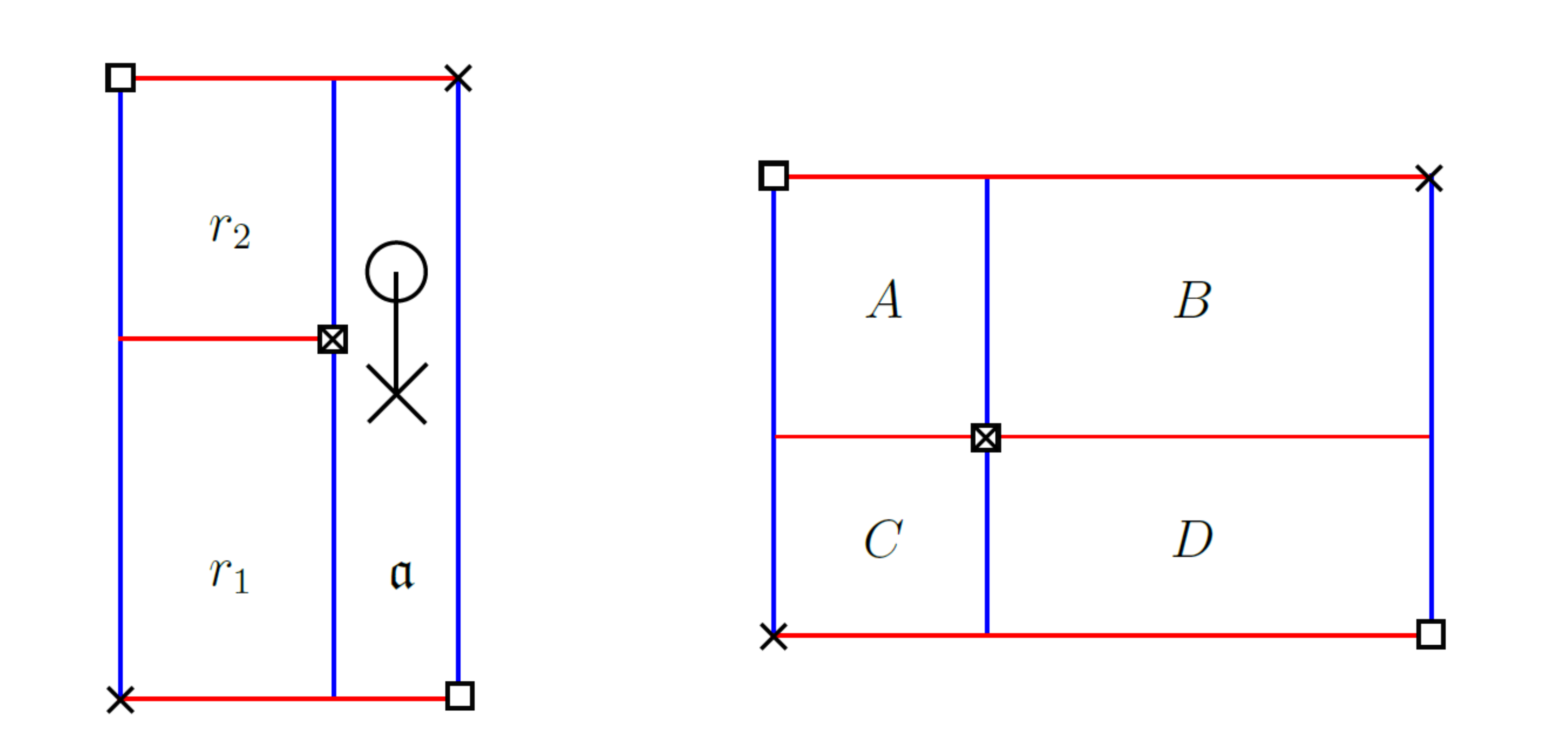}}
\caption{The picture on the left illustrates a Type 2 rectangle. The picture on the right illustrates a rectangle containing a common component of both of the initial and the terminal generators.}
\label{type2}
\end{figure}

\begin{proof}
We show that for each generator $\x$ the sum of contributions from $\partial \circ F$ and $F\circ\partial$ cancel out. From the definition of $F$ and $\partial$ we have to consider various ways a rectangle and a pseudo-domain can be composed. We consider all the cases of the composition of a rectangle and a pseudo-domain $p\in \sigma^F$, and put them into several groups according to the number of the corners they have in common and the type of the rectangle. We begin by listing all the possibilities, then in each case we show that the sum of contributions from $\partial \circ F$ and $F\circ\partial$ cancel out over $\Z_2$. This is an intermediate step in our proof. Then we show that the cancellations hold with $\Z$ coefficients as well.

As we mentioned at the beginning of this subsection, using the commutation move we can assume that $O_1$ is adjacent to the $X$ marking in its row. Depending on whether the $X$ marking is in the left or in the right of $O_1$ we have two cases. If the $X$ marking is in the left of $O_1$ then the only possible $L$-shape region is empty and all the $R$-shape regions are $4m$-gons. Also in this case we do not have any Type 2 rectangles. Therefore the proof for this case is covered in the proof of the more complicated case, where the $X$ marking is in the right of $O_1$. From now on we assume that the $X$ marking is adjacent to $O_1$ and in the right of it.

The rectangle $r$ is either of Type $1$ or $2$. We consider two cases.

\textbf{Case I}: Let $r$ be a Type 1 rectangle.

I(0) There are no common corners between $r$ and $p$, i.e. they are disjoint. The composition can be counted in either way as a term in $\partial \circ F$ or $F\circ\partial$.

I(1) There is one common corner between $r$ and $p$ and the rectangle $r$ does not contain any of the $w_i$'s for $i=1,\dots ,m$, except possibly at the common corner of $r$ and $p$.

I($1'$) Again there is one common corner between $r$ and $p$, but the rectangle $r$ contains exactly one of the $w_i$'s where $0 \leq i \leq m$. There are two possibilities: First, $w_i$ is in the     interior of the right edge of $r$. Second, the lower-right corner of $r$ is $w_i$.
In both cases $r$ has an edge that is not contained in $\t G$ so the composition can only be counted as a term in $F\circ\partial$. To be more precise, since $r$ has an edge that goes through $w_i$ it is not equal to the inverse image of any rectangle from $\t G$ under $\psi$ (see Remark~\ref{psi-inverse}). 

I(2) There are two common corners between $r$ and $p$, other than possibly $w_i$ for some $0 \leq i \leq m$.

I(3) There are three common corners between $r$ and $p$, other than possibly $w_i$ for some $0 \leq i \leq m$. In this case the image of the  $r$ and $p$ in the grid diagram $H$ have at least three corners in common, hence the composite region must contain a whole column or a row. On the other hand, a composite region in $\widetilde{H}$ can not contain any markings other than $X_1$ and $O_1$. Hence the composite region of $r$ and $p$ (in $\widetilde{H}$) contains one of the inverse images of the column in $H$ that contains $X_1$ and $O_1$.

\textbf{Case II : }Let $r$ be a Type 2 rectangle and $p\in \sigma^F$. A Type 2 rectangle contains the $X_1$ and the $O_1$ markings, hence it can only be considered as an empty rectangle in $\widetilde{G}$ (not in $\t H$) and appears as a term in $\partial \circ F$. For $p\in \sigma^F$ the only possible composition is of the form $p*r$. There are two cases.

II(0) The rectangle $r$ of Type 2, does not have any corner in common with $p$.

II(1) $r$ and $p$ shares a corner.
It is not hard to see that in this case, the composite region of $r$ and $p$ (in $\widetilde{H}$) contains one of the inverse images of the column in $H$ that contains $X_1$ and $O_1$.

Now we show that $F$ is a chain map over $\Z_2$. In order to do so we explain which configurations in $\partial \circ F + F \circ \partial$ pair together and cancel out (over $\Z_2$). The proof is similar to Lemma 3.5 from \cite{MOST} and Lemma 3.2 in \cite{D}.

\textbf{I(0)}: This case is obvious, since by changing the order of using the rectangle and the pseudo-domain, we get one contribution form $\partial \circ F$ and one contribution from $F\circ\partial$.

\textbf{I(1)}: A composite region that is formed by a rectangle $r$ and a pseudo-domain $p\in \sigma^F$ which belongs to the case I(1) has two different decompositions. The composition has a unique concave corner that belongs to the boundary of both $r$ and $p$. Cutting horizontally or vertically through the concave corner gives two decompositions. Theses two decompositions can be counted as terms of $\partial \circ F$ or $F\circ\partial$. It is possible that both terms belong to the same composition. These two decompositions cancel each other out. See Figures~\ref{I1General4m3}-\ref{I1General4m6b}.

\textbf{I($\mathbf{1'}$), I(3) and II(1)}: 
It is a simple geometric exercise to show that the contribution from any composite region of type I($1'$) can be paired (hence canceled) by the contribution from a regions of either type I(3) or II(1), and vice versa. This has been illustrated in Figures~\ref{I3LGeneral}-\ref{II1General}, for more details see the proof of Lemma 3.5 of \cite{MOST}.

\textbf{I(2) and II(0)}:
Similar to the previous case, one can show that the contribution from terms of I(2) and II(0) pair together. This occurs mainly in three ways that we illustrate in Figure~\ref{II0General}. 

\begin{remark}\label{Additional-comp}
For any given case in the corresponding figure, we have drawn a rectangle and a pseudo-domain that intersect each other (in a certain way associated to that case) in the first grid. Note that this piece of the pseudo-domain might pass through a cut and go to other grids, and also the pieces of the pseudo-domain in other grids may vary. These pictures should be considered as a schematic image for the aforementioned case. Also note that the possible changes in the pieces of the pseudo-domain do not change the proof, and in this sense we have considered \emph{all the} cases.
\end{remark}

Now we give the argument over $\Z$. In fact we check that the cancellations that we explained over $\Z_2$ still hold with signs. We use the following notations.  

\begin{itemize}
\item As we mentioned in Remark~\ref{signremark}, in order to make it easier to follow the arguments, we use the name of the underlying region to indicate a given pseudo-domain. Note that we keep track of the initial and the terminal generators for each pseudo-domain although it does not appear in this notation.  
\item For a pseudo-domain $p$ as in Figure~\ref{I1General4m3} we write $AB\sqsubset p$. Note that $AB$ is not a domain, since there is a branch cut along its boundary. But with an appropriate finger move we can obtain a domain, which by abuse of notation we denote by $AB$ see Figure~\ref{I1General4m3}. In the first few cases we explicitly draw the finger move to make the arguments easier to follow, and later on in order to make the pictures less complicated we do not draw the finger moves.
\end{itemize}

\begin{remark}
In what follows we will check that given a pseudo-domain $p$ and a rectangle $r$, the above cancellations hold over $\Z$. Using the notation of Definition~\ref{F-def} we have to show that $F^R\circ\partial=-\partial\circ F^R$ and $F^L\circ\partial=\partial\circ F^L$. Depending on the order of the compositions in different cases, these identities translate into the following equalities:
\begin{itemize}
\item If $p,p'\in\sigma^L$ and $p*r=r'*p'$, then $\mu(p)\cdot \SS(r)=\SS(r')\cdot \mu(p')$
\item If $p,p'\in\sigma^R$ and $p*r=r'*p'$, then $\mu(p)\cdot \SS(r)=-\SS(r')\cdot \mu(p')$
\item If $p,p'\in\sigma^F$ and $p*r=p'*r'$ or $r*p=r'*p'$, then $\mu(p)\cdot \SS(r)=-\mu(p')\cdot \SS(r').$
\end{itemize}

We will express $\mu(p)$, the sign of $p$, as the product of signs of certain simpler regions ($R_i^j$'s and $\hh_k$'s). Note that $r$ has intersection with at most one of these regions which we call the \emph{main piece}, and denote it by $main(p)$. The main piece is the only region that we have to analyze carefully, since all the other regions are disjoint from the rectangle $r$ and using the S-3 property they can be swapped with $r$. In the course of the proof we only use the alphabetical notation for the underlying region of the main piece, and the other regions of $p$ are understood from the initial generator.

\end{remark}

\textbf{I(0) over $\Z$:}

 This case follows from the S-3 property of the sign assignment.

\textbf{I(1) over $\Z$:}

In this case all the non-trivial configurations involve $4m$-gons. Since in all other configurations the cancellation over $\Z$ follows from the S-3 property of the sign assignment. The difficulty arises in the case of a $4m$-gon as we used the finger move to define the sign of a $4m$-gon. In this case the cancellation between two different ways that we can cut the composite region does not readily follow from the S-3 property. To summarize, the case of a Type $L$ pseudo-domain follows from the S-3 property, and the case of a general Type $R$ pseudo-domain follows from the case of a $4m$-gon and the S-3 property. So we only discuss the case of a $4m$-gon.

If a rectangle $r$ has a corner in common with a $4m-$gon, we get five possible configurations. See Figures \ref{I1General4m3}-\ref{I1General4m6b}. Using Lemma \ref{indep-lemma}, in each case we use an appropriate finger move to compute the sign of the $4m-$gon.

In the configuration of Figure \ref{I1General4m3}, we use the finger move along $a_1$. In this figure $r=B$ , $main(p)=AC$ and $r'=C$.

\begin{figure}[h]
\centerline{\includegraphics[scale=0.8]{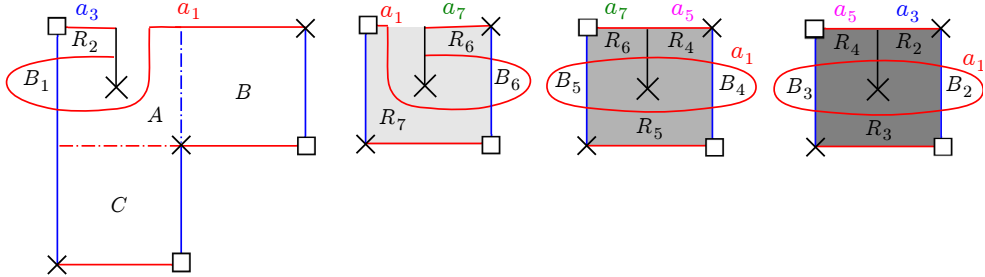}}
\caption {In this picture we illustrate a composite region of type $I(1)$, that has another decomposition of the same type, and the contributions from these two configurations cancel each other out. We have $r*p=r'*\overline{p}$ where $r=B$ , $main(p)=AC$, $r'=C$ and $main(\overline{p})=AB$.}
\label{I1General4m3}
\end{figure}

\begin{eqnarray*}
\SS(r)\cdot \mu(p)
&=& (-1)^{m-1} \SS(r)\cdot \SS(F(R_{1}^{1}))\cdot \SS(F({R'}_{2}^{1}))\cdots \SS(F(R_{2m-1}^{1})) \\
&=& (-1)^{m-1} \SS(B)\cdot \SS(AC)\cdot \SS(F({R'}_{2}^{1}))\cdots \SS(F(R_{2m-1}^{1})) \\
&=& (-1)^{m} \SS(C)\cdot \SS(AB)\cdot \SS(F({R'}_{2}^{1}))\cdots \SS(F(R_{2m-1}^{1})) \\
&=&  (-1)^{m} \SS(r')\cdot \SS(F({\overline{ R}_{1}}^{1}))\cdot \SS(F({R'}_{2}^{1}))\cdots \SS(F(R_{2m-1}^{1}))\\
&=& - \SS(r')\cdot \mu(\overline{p}).
\end{eqnarray*}

\begin{figure}[h]
\centerline{\includegraphics[scale=0.8]{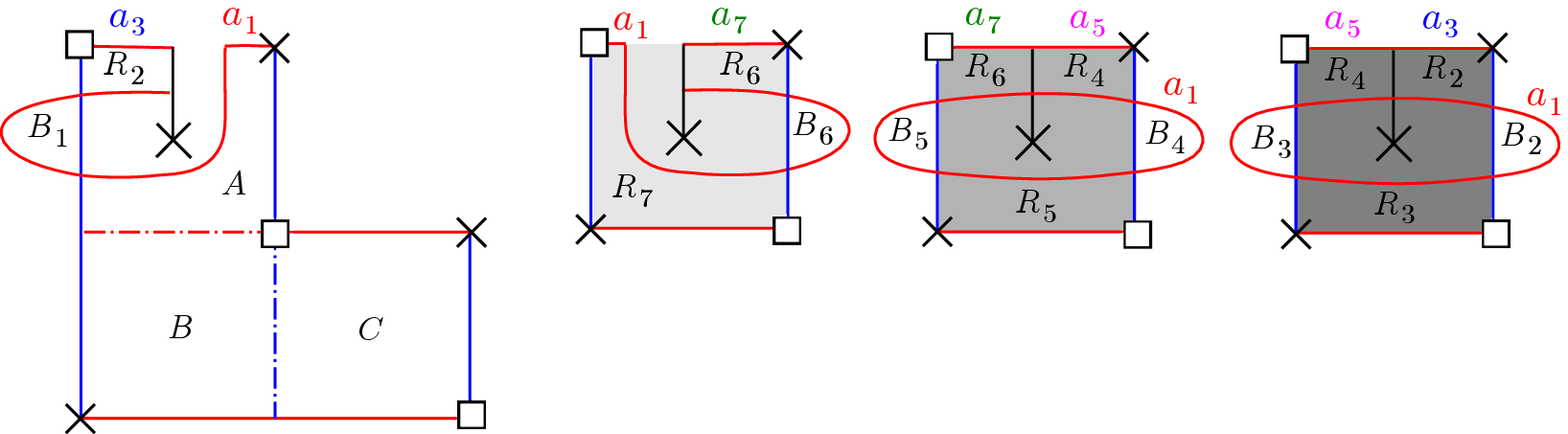}}
\caption {A cancellation of a pair of type $I(1)$ configurations.}
\label{I1General4m5a}
\end{figure}

In the configuration of Figure \ref{I1General4m5a}, we use the finger move along $a_1$. In this configuration we have $p*r=r'*\overline{p}$ where $main(p)=AB$, $r=C$, $r'=BC$ and $main(\overline{p})=A$.

\begin{eqnarray*}
 \mu(p)\cdot \SS(r)
&=& (-1)^{m-1}  \SS(F(R_{1}^{1}))\cdot \SS(F({R'}_{2}^{1}))\cdots \SS(F(R_{2m-1}^{1}))\cdot\SS(r) \\
&=& (-1)^{m-1}  \SS(F(R_{1}^{1}))\cdot\SS(r) \cdot \SS(F({R'}_{2}^{1}))\cdots \SS(F(R_{2m-1}^{1})) \\
&=& (-1)^{m-1}  \SS(AB)\cdot\SS(C) \cdot \SS(F({R'}_{2}^{1}))\cdots \SS(F(R_{2m-1}^{1})) \\
&=& (-1)^{m}  \SS(BC)\cdot\SS(A) \cdot \SS(F({R'}_{2}^{1}))\cdots \SS(F(R_{2m-1}^{1})) \\
&=&  (-1)^{m} \SS(r')\cdot \SS(F({\overline{ R}_{1}}^{1}))\cdot \SS(F({R'}_{2}^{1}))\cdots \SS(F(R_{2m-1}^{1}))\\
&=& - \SS(r')\cdot \mu(\overline{p}).
\end{eqnarray*}

Similarly, in the configuration of Figure \ref{I1General4m5b}, we use the finger move along $a_1$. In this configuration we have $p*r=r'*\overline{p}$ where $main(p)=A$, $r=BC$, $r'=B$ and $main(\overline{p})=AC$.
\begin{figure}[h]
\centerline{\includegraphics[scale=0.8]{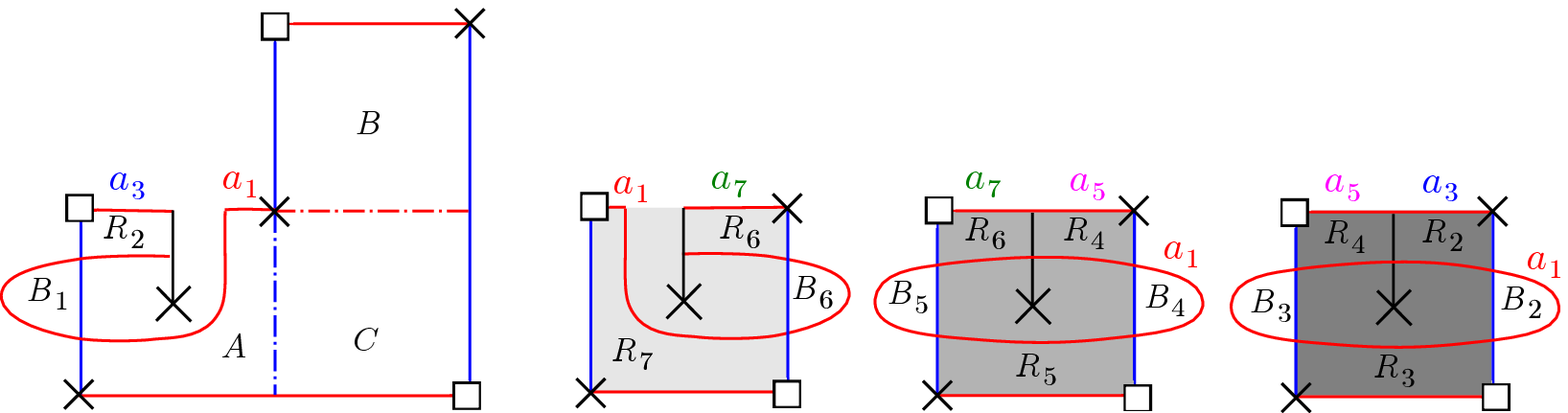}}
\caption {A cancellation of a pair of type $I(1)$ configurations.}
\label{I1General4m5b}
\end{figure}

\begin{eqnarray*}
 \mu(p)\cdot \SS(r)
&=& (-1)^{m-1}  \SS(F(R_{1}^{1}))\cdot \SS(F({R'}_{2}^{1}))\cdots \SS(F(R_{2m-1}^{1}))\cdot\SS(r) \\
&=& (-1)^{m-1}  \SS(F(R_{1}^{1}))\cdot\SS(r) \cdot \SS(F({R'}_{2}^{1}))\cdots \SS(F(R_{2m-1}^{1})) \\
&=& (-1)^{m-1}  \SS(A)\cdot\SS(BC) \cdot \SS(F({R'}_{2}^{1}))\cdots \SS(F(R_{2m-1}^{1})) \\
&=& (-1)^{m}  \SS(B)\cdot\SS(AC) \cdot \SS(F({R'}_{2}^{1}))\cdots \SS(F(R_{2m-1}^{1})) \\
&=&  (-1)^{m} \SS(r')\cdot \SS(F({\overline{ R}_{1}}^{1}))\cdot \SS(F({R'}_{2}^{1}))\cdots \SS(F(R_{2m-1}^{1}))\\
&=& - \SS(r')\cdot \mu(\overline{p}).
\end{eqnarray*}

Using the finger move along the $\A$-curve $a_1$ for the configuration of Figure \ref{I1General4m6a}, in which $main(p)=A$, $r=BC$ ,$r'=B$ and $main(\overline{p})=A$, we obtain the following equalities.

\begin{figure}[h]
\centerline{\includegraphics[scale=0.8]{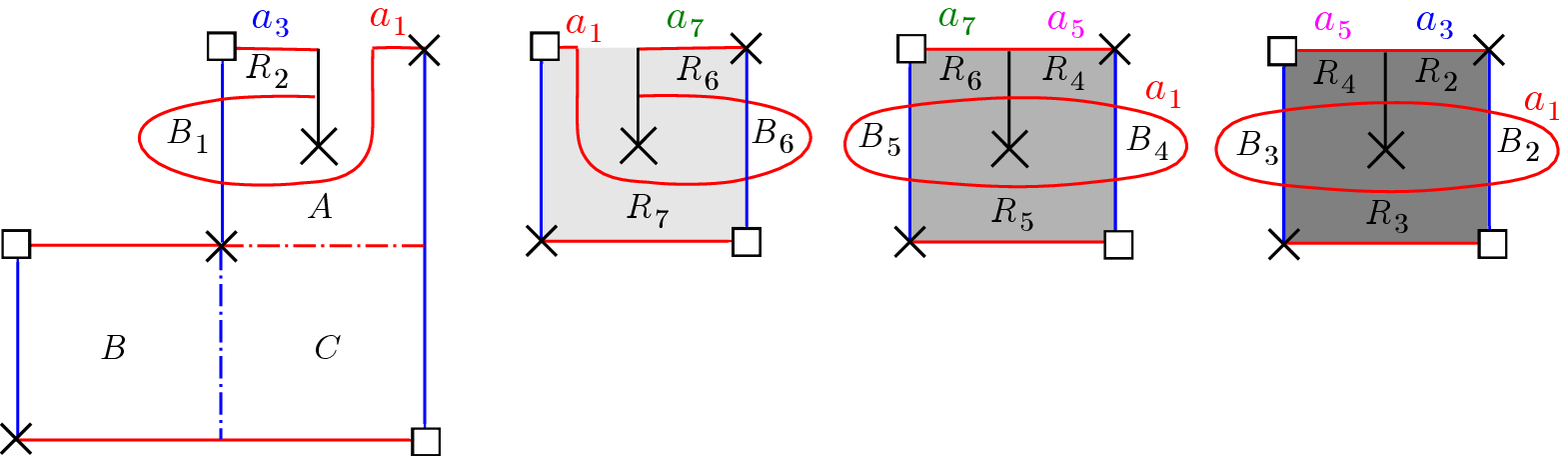}}
\caption {A cancellation of a pair of type $I(1)$ configurations.}
\label{I1General4m6a}
\end{figure}

\begin{eqnarray*}
 \mu(p)\cdot \SS(r)
&=& (-1)^{m-1}  \SS(F(R_{1}^{1}))\cdot \SS(F({R'}_{2}^{1}))\cdots  \SS(F(R_{2m-1}^{1}))\cdot\SS(r) \\
&=& (-1)^{m-1}  \SS(F(R_{1}^{1}))\cdot\SS(r) \cdot \SS(F({R'}_{2}^{1}))\cdots  \SS(F(R_{2m-1}^{1})) \\
&=& (-1)^{m-1}  \SS(A)\cdot\SS(BC) \cdot \SS(F({R'}_{2}^{1}))\cdots  \SS(F(R_{2m-1}^{1})) \\
&=& (-1)^{m}  \SS(B)\cdot\SS(AC) \cdot \SS(F({R'}_{2}^{1}))\cdots  \SS(F(R_{2m-1}^{1})) \\
&=&  (-1)^{m} \SS(r')\cdot \SS(F({\overline{ R}_{1}}^{1}))\cdot \SS(F({R'}_{2}^{1}))\cdots \SS(F(R_{2m-1}^{1}))\\
&=& - \SS(r')\cdot \mu(\overline{p}).
\end{eqnarray*}

Similarly, using the finger move along the $\A$-curve $a_1$ for the configuration of Figure \ref{I1General4m6b}, in which $main(p)=AB$, $r=C$, $r'=BC$ and $main(\overline{p})=A$, we obtain:

\begin{figure}[h]
\centerline{\includegraphics[scale=0.8]{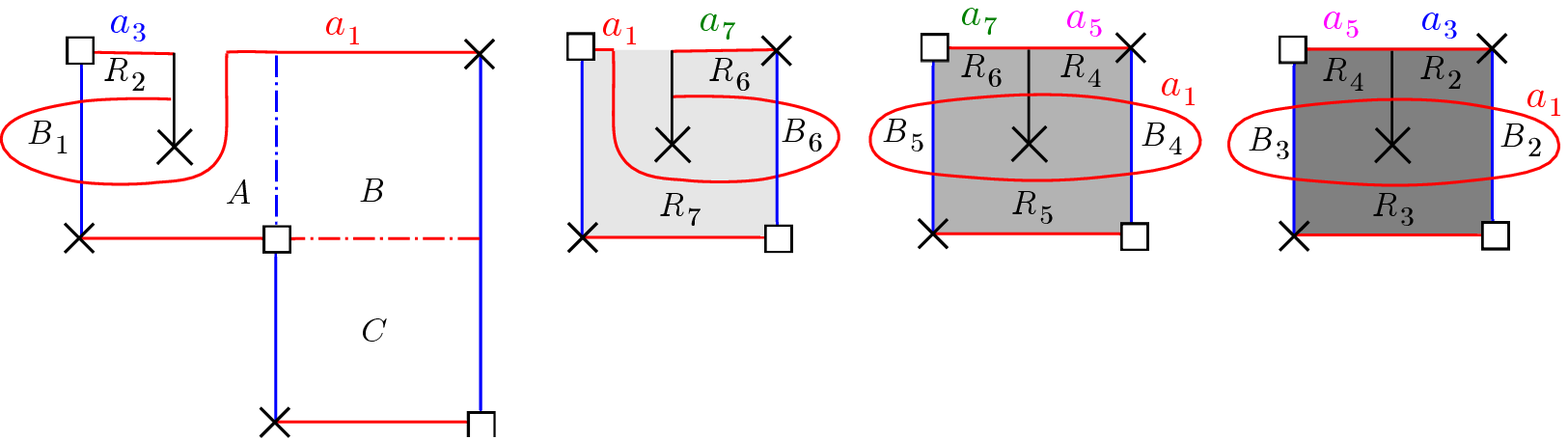}}
\caption {A cancellation of a pair of type $I(1)$ configurations.}
\label{I1General4m6b}
\end{figure}

\begin{eqnarray*}
 \mu(p)\cdot \SS(r)
&=& (-1)^{m-1}  \SS(F(R_{1}^{1}))\cdot \SS(F({R'}_{2}^{1}))\cdots  \SS(F(R_{2m-1}^{1}))\cdot\SS(r) \\
&=& (-1)^{m-1}  \SS(F(R_{1}^{1}))\cdot\SS(r) \cdot \SS(F({R'}_{2}^{1}))\cdots  \SS(F(R_{2m-1}^{1})) \\
&=& (-1)^{m-1}  \SS(AB)\cdot\SS(C) \cdot \SS(F({R'}_{2}^{1}))\cdots  \SS(F(R_{2m-1}^{1})) \\
&=& (-1)^{m}  \SS(BC)\cdot\SS(A) \cdot \SS(F({R'}_{2}^{1}))\cdots  \SS(F(R_{2m-1}^{1})) \\
&=&  (-1)^{m} \SS(r')\cdot \SS(F({\overline{ R}_{1}}^{1}))\cdot \SS(F({R'}_{2}^{1}))\cdots \SS(F(R_{2m-1}^{1}))\\
&=& - \SS(r')\cdot \mu(\overline{p}).
\end{eqnarray*}

\begin{remark}
For a pseudo-domain $p$ as in Figure \ref{I1General4m6b}, we have illustrated an appropriate finger move, that enabled us to compute the sign of $p$ as the product of the signs of various regions. 
In the first few cases we have explicitly drawn the finger move to make the arguments easier to follow, but from this point in order to make the pictures less complicated, we do not draw the modified curves. As a result, we may write the sign of a pseudo-domain as a product of signs of regions with a cut along their boundary. This should be understood as the region that we obtain after the finger move (which has no cut along the boundary). For example look at the region $AB$ in Figure~\ref{I1General4m6b}.

\end{remark}

\textbf{I($\mathbf{1'}$) and I(3) over $\Z$:}

\begin{figure}[h]
\centerline{\includegraphics[scale=0.8]{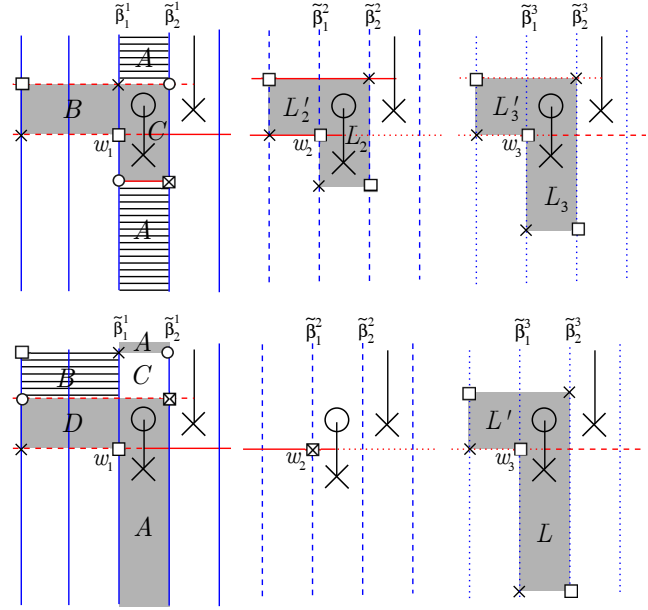}}
\caption {Cancellations of a type I($1'$) configuration with a type I(3) configuration with a Type $L$ pseudo-domain.}
\label{I3LGeneral}
\end{figure}

The contribution from a composite region of type I($1'$) can be paired with a contribution of a composite region of type I(3) or II(1). For a type I(3) composite region we have to consider two cases, depending on whether the pseudo-domain is of Type $L$ or Type $R$. We draw the possible configuration with a Type $L$ pseudo-domain in Figure~\ref{I3LGeneral}, and Type $R$ cases are drawn in Figure~\ref{I3RGeneral}.

In the first picture in Figure \ref{I3LGeneral} we have $r*p=r'*p'$ where $r=A$, main piece of $p$ is $BC$, the main piece of $p'$ is empty and $r'=B$ is a Type 1 rectangle. Note that since an edge of $r'$ is on a lift of $\B_1$ it can only be counted as an empty rectangle in $\t H$. In the following computation according to the S-2 property, the term $\SS(A)\cdot \SS(C)$ is equal to $-1$.

%Note that $p'$ is not necessarily trivial, since it might have components on other grids. 

\begin{eqnarray*}
\SS(r) \cdot \mu(p)
&=&  \SS(A)\cdot \SS(\hh_1) \cdot \SS(\hh_2) \cdot\cdots \SS(\hh_k) \\
&=&  \SS(A)\cdot \{\SS(C)\cdot \SS(B) \}  \cdot \SS(\hh_2) \cdot\cdots \SS(\hh_k) \\
&=&  \{\SS(A)\cdot \SS(C)\} \cdot \SS(B)  \cdot \SS(\hh_2) \cdot\cdots \SS(\hh_k) \\
&=& -\SS(B) \cdot \SS(\hh_2) \cdot\cdots \SS(\hh_k) \\
&=&-\SS(r') \cdot \mu(p').
\end{eqnarray*}

In the second picture of Figure~\ref{I3LGeneral}, $p*r=r'*p'$ where $main(p)=AD$, $r=BC$, $r'=BD$ is a Type 1 rectangle and the main piece of $p'$ is empty. Note that as in the last case, an edge of $r'$ is an arc on a lift of $\B_1$ so it can only be counted as an empty rectangle in $\t H$. 

\begin{eqnarray*}
\mu(p) \cdot \SS(r)
&=&\SS(\hh_1) \cdot\cdots \SS(\hh_k)\cdot\SS(BC)\\
&=&\SS(A)\cdot \SS(D) \cdot \SS(\hh_2) \cdot\cdots \SS(\hh_k) \cdot \SS(BC)\\
&=&\SS(A)\cdot \SS(D) \cdot \SS(BC) \cdot\SS(\hh_2) \cdot\cdots \SS(\hh_k)\\
&=&-\SS(A)\cdot \SS(C) \cdot \SS(BD) \cdot \SS(\hh_2) \cdot\cdots \SS(\hh_k)\\
&=& \SS(BD) \cdot \SS(\hh_2) \cdot\cdots \SS(\hh_k)\\
&=&\SS(r') \cdot \mu(p').
\end{eqnarray*}

In the second equality we swap the terms $\SS(\hh_2) \cdot\cdots \SS(\hh_k)$ with $\SS(BC)$. The sign does not change since the sign of each $\hh_i$ is the product of the signs of two formal rectangles, and the S-3 property has been used an even number of times. The term $\SS(A)\cdot \SS(C)$ is equal to $-1$ according to the S-2 property.

%%%%%%%%%%%%%%%%%%%%%%%%%%%%%%%%%
%%%%%%%%%%%%%%%%%%%%%%%%%%%%%%%%%
%%%%%%%%%%%%%%%%%%%%%%%%%%%%%%%%%
%%%%%%%%%%%%%%%%%%%%%%%%%%%%%%%%%

\begin{figure}[h]
\centerline{\includegraphics[scale=0.8]{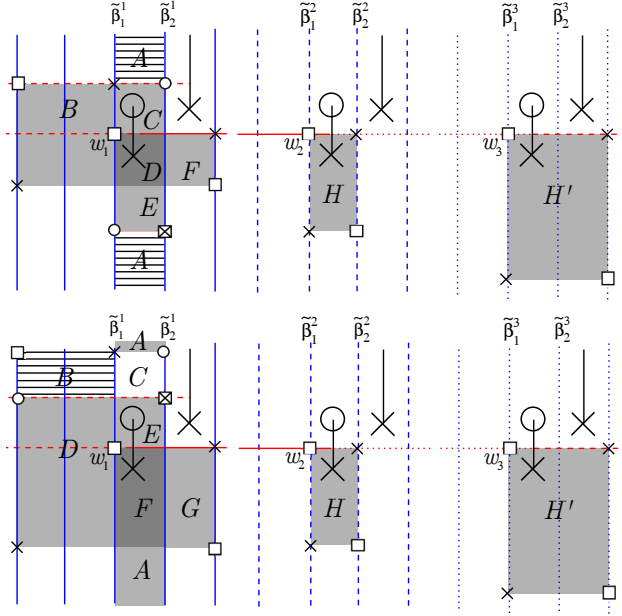}}
\caption {Cancellations of a type I($1'$) configuration with a type I(3) configuration with a Type $R$ pseudo-domain.}
\label{I3RGeneral}
\end{figure}

Now we consider the cases of I(3) configurations with a Type $R$ pseudo-domain. In the first picture in Figure~\ref{I3RGeneral}, $r*p=r'*p'$ where $r=A$, $main(p)=BCDEF$, $r'=B$ and the main piece of $p'$ is $DF$. In the second equality we use the S-2 property to replace the term $\SS(A)\cdot \SS(CDE)$ with $-1$. 

\begin{eqnarray*}
\SS(r) \cdot \mu(p)
&=& \{ \SS(A)\cdot \SS(CDE) \} \cdot \SS(B) \cdot \mu(p')\\
&=& - \SS(B) \cdot \mu(p')\\
&=&-\SS(r') \cdot \mu(p').
\end{eqnarray*}

In the second picture in Figure~\ref{I3RGeneral}, $p*r=r'*p'$ where
$main(p)=ADEFG$, $r=BC$, $r'=BD$ and $main(p')=FG$.

\begin{eqnarray*}
\mu(p) \cdot \SS(r)
&=&\SS(AEF)\cdot \SS(D) \cdot \mu(p'') \cdot \SS(BC)\\
&=&-\SS(AEF)\cdot \{\SS(D)\cdot \SS(BC)\} \cdot \mu(p') \\
&=&\SS(AEF)\cdot \SS(C)\cdot \SS(BD) \cdot \mu(p') \\
&=&-\SS(r') \cdot \mu(p')
\end{eqnarray*}

Note that $p'$ and $p''$ are pseudo-doamins of Type $R$ with the same underlying main components, but their initial generators are not the same. In the second equality we switch the order of a rectangle with underlying domain $BC$ with a disjoint Type $R$ pseudo-domain $p''$ with $main(p'')=FG$, so we get a minus sign. In fact the sign of each Type $R$ pseudo-domain is defined as the product of the signs of an odd number of formal rectangles.
In the forth equality we use the S-2 property to write $\SS(AEF) \cdot \SS(C)=-1$.

%%%%%%%%%%%%%%%%%%%%%%%%%%%%%%%%%
%%%%%%%%%%%%%%%%%%%%%%%%%%%%%%%%%
%%%%%%%%%%%%%%%%%%%%%%%%%%%%%%%%%
%%%%%%%%%%%%%%%%%%%%%%%%%%%%%%%%%

\textbf{I($\mathbf{1'}$) and II(1) over $\Z$:} 

Here we discuss the pairing of the contributions from the configurations of type I($1'$) and II(1). In the first picture of Figure~\ref{II1General}, $p*r=r'*\overline{p}$ where $main(p)=AE$, $r=BCDE$ is a Type 2 rectangle, $r'=BD$ and $main(\overline{p})=E$. Using the S-2 property, in the fourth line we substitute $\SS(A)\cdot\SS(CE)$ with $-1$.

\begin{eqnarray*}
\mu(p) \cdot \SS(r) 
&=& (-1)^{m-1}\SS(F(R_{1}^{1}))\cdots \SS(F({R'}_{2m-2}^{1}))\cdot \SS(F(R_{2m-1}^{1})) \cdot \SS(r) \\
&=& (-1)^{m-1}\SS(F(R_{1}^{1}))\cdots \SS(F({R'}_{2m-2}^{1}))\cdot \SS(AE) \cdot \left[   \SS(D)\cdot\SS(CE)\cdot \SS(B) \right] \\
&=& (-1)^{m}\SS(F(R_{1}^{1}))\cdots \SS(F({R'}_{2m-2}^{1}))\cdot \SS(DE) \cdot  \SS(A)\cdot\SS(CE)\cdot \SS(B) \\
&=& (-1)^{m-1}\SS(F(R_{1}^{1}))\cdots \SS(F({R'}_{2m-2}^{1}))\cdot \SS(DE) \cdot \SS(B) \\
&=& (-1)^{m}\SS(F(R_{1}^{1}))\cdots \SS(F({R'}_{2m-2}^{1}))\cdot \SS(BD) \cdot \SS(E) \\
&=& (-1)^{m}\SS(BD) \cdot \SS(F(R_{1}^{1}))\cdots \SS(F({R'}_{2m-2}^{1}))\cdot  \SS(E) \\
&=& - \SS(r') \cdot \mu(\overline{p}).
\end{eqnarray*}

\begin{figure}[h]
\centerline{\includegraphics[scale=0.8]{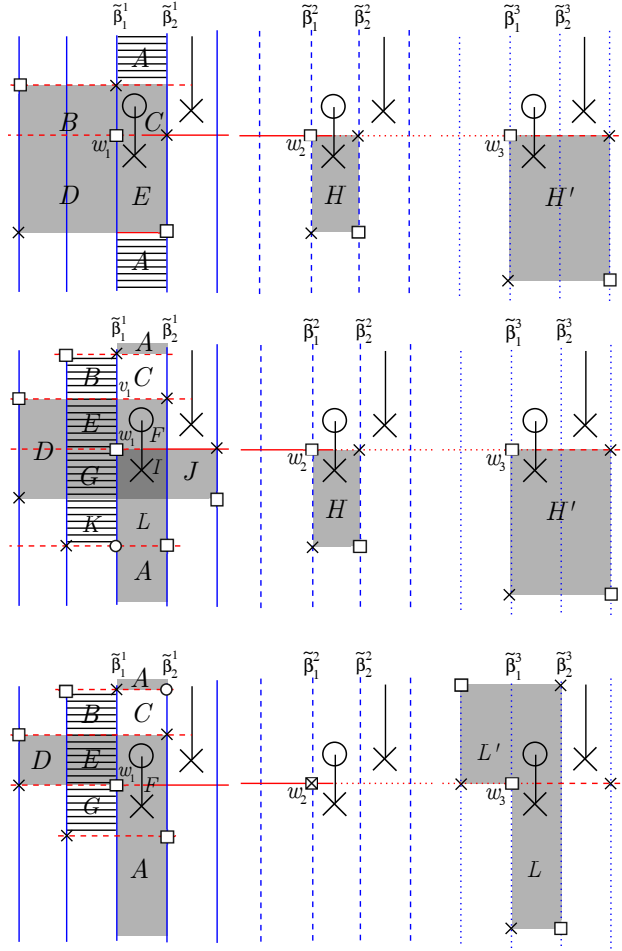}}
\caption {Cancellations of a type I($1'$) configuration with a type II(1) configuration.}
\label{II1General}
\end{figure}

%%%%%%%%%%%%%%%%%%%%%%%%%%%%%%%%%%%%
%%%%%%%%%%%%%%%%%%%%%%%%%%%%%%%%%%%%

In the second picture of Figure~\ref{II1General}, $p*r=r'*\overline{p}$ where $main(p)=ADEFGIJL$, $r=BCEFGIKL$ is a Type 2 rectangle, $r'=BEGK$ and $main(\overline{p})=DEFGIJL$. We denote by $p'$ the Type $R$ pseudo-domain with the main part $IJ$ and the remaining parts equal to the other parts of $p$ (in other grids).

In the second equality, note that $p'$ is disjoint from the formal rectangle associated with $r$. Using S-3 property an odd number of times leads to a minus sign and switches the order of them. Once again, in the third equality, the formal rectangle associated with $DEG$ is disjoint from the formal rectangle associated with $BCEFGIKL$, hence using the S-3 property switches their order and gives a minus sign. In the forth equality the Type 2 rectangle with underlying pseudo-domain $BCEFGIKL$ has its interior component at $v_1$. Hence we use the expression in Remark~\ref{type2-sign} to compute the sign of the associated rectangle. We use the S-2 property to substitute the term $\SS(AFIL) \cdot \SS(C)$ with $-1$.

\begin{eqnarray*}
\mu(p) \cdot \SS(r)
&=&\SS(AFIL)\cdot \SS(DEG) \cdot \mu(p')\cdot \SS(BCEFGIKL)\\
&=&-\SS(AFIL)\cdot \left[\SS(DEG) \cdot \SS(BCEFGIKL)\right] \cdot \mu(p')\\
&=&\SS(AFIL) \cdot \SS(BCEFGIKL) \cdot \SS(DEG) \cdot \mu(p')\\
&=&\left[ \SS(AFIL) \cdot \SS(C) \right] \cdot \SS(BEGK) \cdot \SS(FIL) \cdot \SS(DEG) \cdot \mu(p')\\
&=&-\SS(r') \cdot \mu(p')
\end{eqnarray*}

%%%%%%%%%%%%%%%%%%%%%%%%%%%%%%%%%%%%
%%%%%%%%%%%%%%%%%%%%%%%%%%%%%%%%%%%%

In the third picture in Figure~\ref{II1General}, $p*r=r'*\overline{p}$ where $main(p)=ADEF$, $r=BCEFG$ is a Type 2 rectangle, $r'=BEG$ and $main(\overline{p})=DEF$. In the second equality we moved the term $\SS(\hh_2) \cdot\cdots \SS(\hh_k)$ to the end, and this does not change the sign as we used property S-3 an even number of times.

\begin{eqnarray*}
\mu(p) \cdot \SS(r)
&=&\SS(AF)\cdot \SS(DE) \cdot \SS(\hh_2) \cdot\cdots \SS(\hh_k) \cdot \SS(G) \cdot \SS(CF) \cdot \SS(BE)\\
&=&\SS(AF)\cdot \left[ \SS(DE) \cdot \SS(G) \right] \cdot \SS(CF) \cdot \SS(BE) \cdot \SS(\hh_2) \cdot\cdots \SS(\hh_k)\\
&=&-\SS(AF)\cdot \SS(EG) \cdot \left[ \SS(D) \cdot \SS(CF) \right] \cdot \SS(BE) \cdot \SS(\hh_2) \cdot\cdots \SS(\hh_k)\\
&=&\SS(AF)\cdot \SS(EG) \cdot \SS(CF) \cdot \left[ \SS(D) \cdot \SS(BE) \right] \cdot\SS(\hh_2) \cdot\cdots \SS(\hh_k)\\
&=&-\SS(AF)\cdot \SS(EG) \cdot \left[ \SS(CF) \cdot \SS(B) \right] \cdot \SS(DE) \cdot\SS(\hh_2) \cdot\cdots \SS(\hh_k)\\
&=&\SS(AF)\cdot \left[ \SS(EG) \cdot \SS(BC)\right] \cdot \left[ \SS(F) \cdot \SS(DE) \cdot \SS(\hh_2) \cdot\cdots \SS(\hh_k) \right] \\
&=&- \left[ \SS(AF)\cdot \SS(C) \right] \cdot \SS(BEG) \cdot \mu(\overline{p})\\
&=&\SS(r') \cdot \mu(\overline{p})
\end{eqnarray*}

%%%%%%%%%%%%%%%%%%%%%%%%%%%%%%%%%%%%%%%%%%%%
%%%%%%%%%%%%%%%%%%%%%%%%%%%%%%%%%%%%%%%%%%%%
%%%%%%%%%%%%%%%%%%%%%%%%%%%%%%%%%%%%%%%%%%%%

\textbf{I(2) and II(0) over $\Z$}:
In the first picture in Figure \ref{II0General}, $p*r=p'*r'$ where $p$ is a Type $R$ pseudo-domain whose main piece is $ABCDEF$, $r=G$ is a Type 1 rectangle, $main(p')=DEFG$ and $r'=ABCD$ is a Type 2 rectangle. We denote by $p_0$ (resp $p_1$) the Type $R$ pseudo-domain whose main part is $DE$ (resp $D$) and its remaining parts are equal to the other parts of $p$ (in other grids).

In the second equation a formal rectangle associated with $G$ is switched with the pseudo-domain $p_0$. The configuration is identical to the one in Figure~\ref{I1General4m6b}, and the same argument works here. 
In the third equation we use the S-3 property. In this equation we still do not have a component at $w_1$ so the rectangle $ABCD$ is not a Type 2 rectangle. In the forth equation, we expand the sign of the pseudo-domain $p_1$. 
In the fifth equation we use the identity of Remark~\ref{type2-sign}.
In the sixth equation the formal rectangle associated with $ABCD$ is disjoint from the pseudo-domain with the main piece $DEFGHH'$ and we switch them. 

\begin{eqnarray*}
\mu(p)\cdot \SS(r)
&=& \SS(BDF)\cdot \SS(AC) \cdot \mu(p_0) \cdot \SS(G) \\
&=& -\SS(BDF)\cdot \SS(AC)\cdot \SS(EG)\cdot \mu(p_1)\\
&=&\SS(F(ABCD))\cdot \SS(F)\cdot \SS(EG)\cdot \mu(p_1)\\
&=& \SS(F(ABCD))\cdot \SS(F)\cdot \SS(EG)\cdot \\
& & (-1)^{m-1}  \SS(F({R}_{1}^{1}=D))\cdot \SS(F({{ R}'}_{2}^{1}))\cdots  \SS(F({R}_{2m-1}^{1}))\\
&=& (-1)^{m-1} \SS(F(ABCD)) \cdot \SS(F({\overline{R}}_{1}^{1}=DEFG))\cdot \SS(F({{ R}'}_{2}^{1}))\cdots  \SS(F({R}_{2m-1}^{1}))\\
&=&\SS(F(ABCD)) \cdot \mu(p')\\
&=&- \mu(p') \cdot \SS(r').
\end{eqnarray*}

\begin{figure}[h]
\centerline{\includegraphics[scale=0.8]{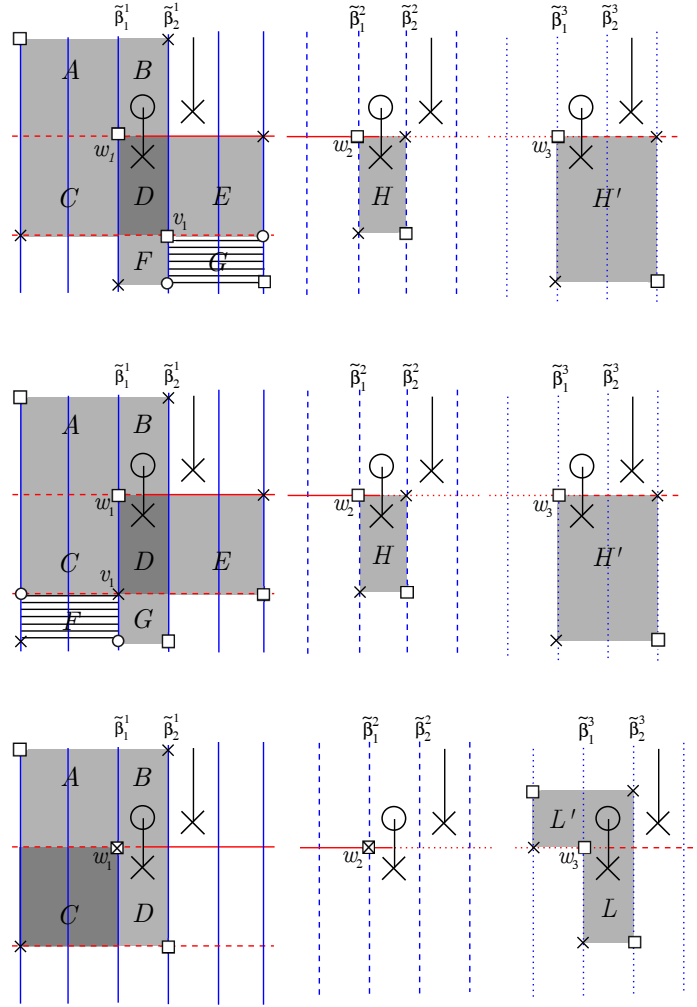}}
\caption {Cancellations of a type I(2) configuration with a type II(0) configuration.}
\label{II0General}
\end{figure}

In the second picture in Figure \ref{II0General}, $r*p=p'*r'$ where $r=F$ is a Type 1 rectangle, $p$ has the main peice $ABCDEG$, and the main piece of $p'$ is $DE$ and $r'=ABCDFG$ is a Type 2 rectangle.

In the second equality $ABCDFG$ is a rectangle that has a component of its initial generator at $v_1$. We used the identity of Remark~\ref{type2-sign} to compute its sign. In the Third equation, we switched the order of the two terms and $r'=ABCDFG$ is a Type 2 rectangle.

\begin{eqnarray*}
\SS(r) \cdot \mu(p)
&=&\SS(F)\cdot \SS(BDG)\cdot \SS(AC)\cdot \mu(p')\\
&=&\SS(F(ABCDFG))\cdot \mu(p')\\
&=&-\mu(p') \cdot \SS(r').
\end{eqnarray*}

In the third picture in Figure \ref{II0General}, $r*p=p'*r'$ where $r=C$, $p$ is of Type $L$ with $main(p)=ABD$, $r'=ABCD$ is a Type 2 rectangle and $p'$ has no intersection with $r'$ (in other words its main component is empty). Note that $C$ can only be counted as a rectangle in $\widetilde{H}$ because one of its edges has an arc on a lift of $\B_1$.

In the second equality we move the two terms $\SS(\hh_2) \cdot\cdots \SS(\hh_k)$ to the beginning. For this move, we used the S-3 property  an even number of times and the sign does not change.

\begin{eqnarray*}
\SS(r) \cdot \mu(p)
&=&\SS(C)\cdot \SS(BD)\cdot \SS(A) \cdot \SS(\hh_2) \cdot\cdots \SS(\hh_k)\\
&=&\SS(\hh_2) \cdot\cdots \SS(\hh_k) \cdot \SS(C)\cdot \SS(BD)\cdot \SS(A)\\
&=& \SS(\hh_2) \cdot\cdots \SS(\hh_k) \cdot \SS(ABCD)\\
&=& \mu(p')\cdot \SS(r').
\end{eqnarray*}

This completes the proof of all the possible cases that we described at the beginning of the proof. 

\end{proof}
 
%%%%%%%%%%%%%%%%%%%%%%%%%%%%%%%%%%%%
%%%%%%%%%%%%%%%%%%%%%%%%%%%%%%%%%%%%
%%%%%%%%%%%%%%%%%%%%%%%%%%%%%%%%%%%%
%%%%%%%%%%%%%%%%%%%%%%%%%%%%%%%%%%%%
\subsubsection{The $\QQ$ filtration}

In order to show that $F$ is a quasi-isomorphism, following the ideas of \cite{MOST}, we define a filtration.

Denote by $\S(\t H,k)$ the set of generators $\x\in S(\t H)$ with Alexander gradings equal to $k\in \Z$. Let $C(\t H,k)$ be the complex generated by elements of $\S(\t H,k)$. Since $\partial$ preserves the Alexander grading $C(\t H,k)$ is a sub-complex and a summand of $C(\t H)$.

Let $\Pi:\t H\rightarrow H$ be the $m$-sheeted branched cover map and $Q\subset H$ be the set consisting of one point in each square of $H$ other than those in the same row or the same column as $\Pi(O_1)$.
Let $\x,\y\in \S(\t H)$ and $p\in \sigma(\x,\y)$. The image of $p$ under $\Pi$ is a 2-chain (not necessarily a domain) in the Heegaard diagram $H$. We denote by $X_1(p)$ (respectively $O_1(p)$), the number of times that $\Pi(X_1)$ (respectively $\Pi(O_1)$) lies in $\Pi(p)$ counted with multiplicity.

We need the following lemmas from \cite{D}, although the context of that paper is for two sheeted coverings, but the proofs work verbatim in the $m$-sheeted case as well.

\begin{lemma}\cite[Lemma 3.16]{D}\label{2chain}
Let $D$ be a 2-chain (not necessarily a domain) in the Heegaard diagram $H$ that contains no basepoints, and $\partial D$ consists of $\A$- and $\B$-circles. Then $D$ is trivial. ($D$ might contain a basepoint in its interior, but the multiplicity of that point is zero.)
\end{lemma}

For fixed $\x , \y \in \S(\t H,k)$, let $p, p' \in \sigma (\x , \y )$ be two pseudo-domains such that $O_1(p)=X_1(p)=O_1(p')=X_1(p')$ and there are no other basepoints inside them. The $2$-chain $\Pi(p)-\Pi(p')$ satisfies the conditions of Lemma~\ref{2chain}, therefore the $\Pi(p)-\Pi(p')$ is trivial. This implies that, counting with multiplicity, we have: $$\# (Q\cap \Pi(p))=\# (Q\cap \Pi(p')).$$

\begin{definition}
We call a pseudo-domain $p$ a \emph{Q-fine} pseudo-domain if its multiplicity at each base point except (possibly) at $X_1$ is zero. 
\end{definition}

\begin{lemma}\cite{D}
Let $Q$ be as above. There exist a function $\QQ$ such that
\begin{equation}\label{QQ}
\QQ(\x)-\QQ(\y)=\# (Q\cap \Pi(p))
\end{equation}
\noindent for all Q-fine pseudo-domains $p \in \sigma (\x, \y)$.
\end{lemma}
\begin{proof}
See \cite{D} page 35.
\end{proof}

\begin{notation}
Using $\QQ$ we can define a filtration on $C(\t H,k)$. The boundary operator of the associated graded object counts those rectangles 
which do not contain any basepoints and does not contain any points of $Q$. We denote by $C_Q$ the associated graded object, and omit the index $k$ when there is no confusion. Similarly we use $C_Q^{\omega}$ for $\omega \in \I_m$ (see Notation~\ref{I-index}) to denote the associated module.

\end{notation}

In order to study the boundary map between different submodules and calculate the homology groups, we need a new definition. 

\begin{definition}
Given $\x, \y \in S(\t H)$ and an empty rectangle $r \in Rect(\x,\y)$ that has no points of $Q$ inside it. We call such rectangle $r$ an \emph{undone} rectangle with respect to $\x$ and a \emph{done} rectangle with respect to $\y$. See Fig.~\ref{CIJNGeneral}.

Each term in the $\partial\x$ is obtained by picking an undone rectangle $r_0 \in Rect(\x,\y)$. We refer to the generator $\y$, as \emph{the generator obtained from  $\x$ by using $r_0$}.
\end{definition}

\begin{lemma}\label{regular}
Let $\x \in C_Q^{\omega}$ with $\omega\in \I_m$. If $\omega\neq (I,I,\cdots,I)$ and $\omega \neq (J,J,\cdots,J)$ then there is at least one done or undone rectangles with respect to $\x$. We call such an element a \emph{regular} element, and the rest of generators are called special, i.e. those in $(I,I,\cdots,I)$ and $(J,J,\cdots,J)$.

\end{lemma}

\begin{proof}
The proof is obvious from the definition.
\end{proof}

\begin{lemma} \label{Homology}
$H_*(C_Q)$ is isomorphic to free module generated by elements of type $(I,I,\cdots,I)$ and $(J,J,\cdots,J)$, where the ring of coefficients can be $\Z$ or $\Z_2$. 
\end{lemma}

\begin{proof}

There are two cases depending on whether $X_2$ is placed in the square, just to the left or just to the right of $O_1$, in the grid diagram $H$ for $K$. 

\textbf{Case (1):} Suppose $X_2$ is in the square just to the left of the square marked $O_1$. Let $\x \in C_Q^{\omega}$, and let $\y$ be in $\partial\x$. In this case if the $i^{th}$ entry of $\x$ is $I$ (resp. $N$), the $i^{th}$ entry of $\y$ remains the same, i.e. $I$ (resp. $N$). Also if the entry of $\x$ is $J$, in $\y$ that entry might remain $J$ or change to $N$. See Figure~\ref{CIJNGeneral}. 

\begin{figure}[h]
\centerline{\includegraphics[scale=0.8]{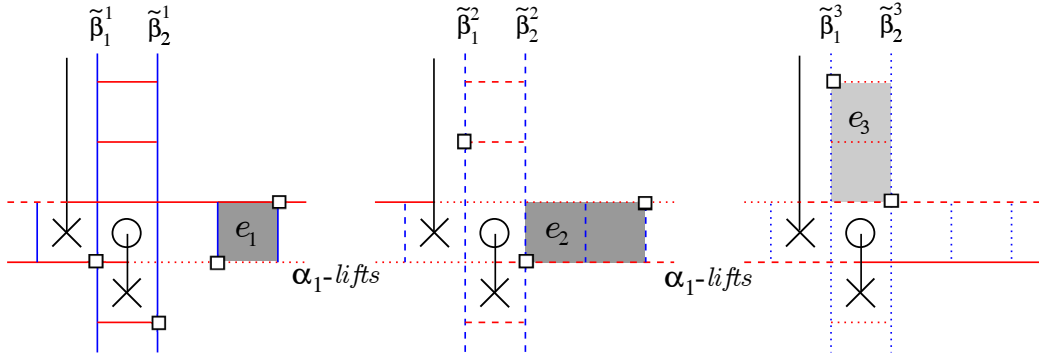}}
\caption {The hollow squares show a generator $\x \in {C_Q}^{IJN}$. The rectangles $e_1$ and $e_2 $ are undone with respect to $\x$. The rectangle $e_3$ is done with respect to $\x$.
Using rectangle 1 leads to a generator in ${C_Q}^{IJN}$. If we use rectangle 2 the result is a generator in ${C_Q}^{INN}$.
Note that
$\A_1^1$ is drawn with a solid line, $\A_1^2$ is drawn with a dashed line and
$\A_1^3$ is drawn with a dotted line.}
\label{CIJNGeneral}
\end{figure}

Let $1 \leq k \leq m$ be an integer, by the above argument we have the following exact sequence:

\begin{center}
\exact{C_Q^{J,\leq k-1}}{C_Q^{J,\leq k}}{C_Q^{J,k}} .
\end{center}

By lemma \ref{regular} for any regular element $\x$ in ${C_Q^{J,k}}$ there exist at least one done or undone rectangle with respect to $\x$. As an element in the quotient complex ${C_Q^{J,k}}$, the boundary of $\x$ consists of those $\y$ that are obtained from $\x$ by using an undone rectangle.

In the Heegaard diagram in Fig.~\ref{CIJNGeneral}, there are exactly two disjoint undone rectangles and one done rectangle, with respect to $\x$. We represent $\x$ by its undone rectangles, $e_1\wedge e_2$.
In each term in $\partial \x$ in ${C_Q}$, one of these undone rectangles (with respect to $\x$) will be counted. For example, we denote by $ e_1$, the term that comes from using $e_2$. In other words, it represents a generator $\y$ such that $e_1$ is the only undone rectangle with respect to $\y$. Using the sign assignment we get the term $e_1$ with the sign $s_2$ of the formal rectangle associated to $e_2$ with the initial generator $\x$.

Hence $\partial \x$ in ${C_Q}$ can be represented by the following notation (we call this the induced boundary operator):

$$\partial  (e_1\wedge e_2) = s_1\cdot e_2 + s_2\cdot e_1.$$

It is not hard to see that since the sign assignment has the property S-3, with an appropriate gauge transform, we can change the sign assignment such that we have:
$$\partial (e_{n_1}\wedge \cdots \wedge e_{n_k})= \displaystyle \sum (-1)^{i}  (e_{n_1}\wedge \cdots \wedge \widehat{e_{n_i}}\wedge \cdots\wedge e_{n_k}) .$$

One can see that each generator in ${C_Q}$ can have at most $2m$ undone rectangles. Given a generator $\x$, we construct a set of generators associated with $\x$. This set consists of the generators obtained from $\x$ using a number of undone rectangles, and also those generators that $\x$ is obtained from them by using a number of undone rectangles. By the above notation, this set with the induced boundary operator from $C_Q$ is isomorphic to an exterior algebra over a vector space of dimension at most $2m$ with coefficients in $\Z$.

In the above example, the induced boundary map after the required change of basis (Using S-3) is represented as follows: 
\[
\partial (e_1 \wedge e_2) = e_1 - e_2
\]

This shows that we can partition our complex as a union of subcomplexes, each isomorphic to the exterior algebra over a vector space with coefficients in $\Z$. Since this complex is isomorphic to the complex of the reduced homology of an $n$-simplex over $\Z$, its homology is trivial.

Hence we get a decomposition of regular elements of $C_Q^{J,k}$ as the direct sum of terms with zero homology. So the homology of $C_Q^{J,k}$ is generated by the special element. Therefore in order to complete the proof we have to show that for any special element $\x$ we have $\partial \x =0$. We consider two cases separately:

\begin{enumerate}[(i)]
\item Let $\x \in (I,I,\cdots,I)$ then by definition of empty rectangles $\partial \x =0$.
\item Let $\x \in (J,J,\cdots,J)$, then any elements in the $\x$ has fewer number of $J$'s, hence as an element of the quotient complex is zero.
\end{enumerate}

\textbf{Case (2):} Suppose that $X_2$ is just to the right of $O_1$, in the grid diagram $H$ for $K$. Let $\x \in C_Q^{\omega}$, and let $\y$ be in $\partial\x$. In this case if the $i^{th}$ entry of $\x$ is $I$ (respectively $J$), the $i^{th}$ entry of $\y$ remains the same, i.e. $I$ (resp. $J$). Also if the entry of $\x$ is $N$, in $\y$ that entry might remain $N$ or change to $I$. See Figure~\ref{CIJNGeneralRight}.

\begin{figure}[h]
\centerline{\includegraphics[scale=0.8]{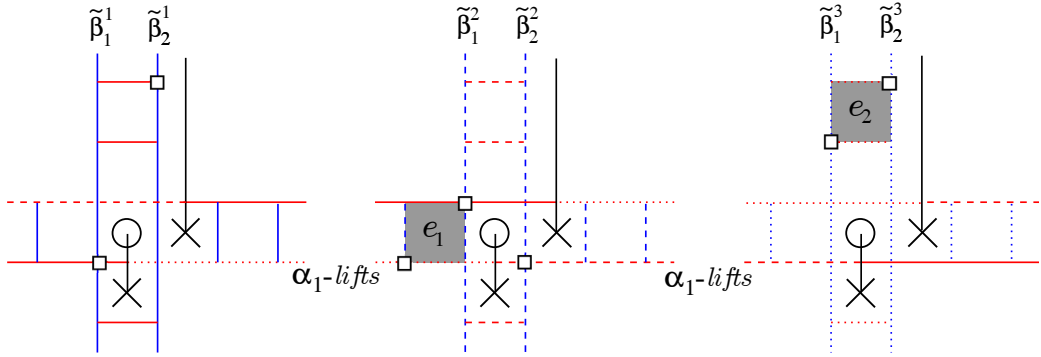}}
\caption {The hollow squares show a generator $\x \in {C_Q}^{IJN}$. We count each of the rectangles shown above in the boundary map. Using rectangles 1 leads to a generator in ${C_Q}^{IJI}$. Using rectangle 2 results in a generator in ${C_Q}^{IJN}$.}
\label{CIJNGeneralRight}
\end{figure}

Let $1 \leq k \leq m$ be an integer, by the above argument we have the following exact sequence:

\begin{center}
\exact{C_Q^{N,\leq k-1}}{C_Q^{N,\leq k}}{C_Q^{N,k}} .
\end{center}

The same argument as in case (1) shows that the homology of regular terms in $C_Q^{N,k}$ vanishes. Hence we have to check that boundary of any special element is zero. In this case the boundary of any special element is trivially zero, since there are no empty rectangle staring from such an element.

\end{proof}

\begin{proposition} \label{quasi}
The map $F$ is a filtered quasi-isomorphism.
\end{proposition}

Note that in order to define the function $\QQ$ for a generator $\x$ of $C(\t G)$, we consider $\psi(\x) \in C(\t H)$ and define $\QQ(\x):= \QQ(\psi(\x))$. For an element $(\x_1,\cdots,\x_m) \in C'$ we define $$\QQ(\x_0,\x_1):=max\{\QQ(\x_1),\cdots,\QQ(\x_m) \}.$$
 
\begin{proof}
First we show that $F$ preserve the $Q$-filtration. Let $\x$ be a generator of $C(\t H)$. We want to show $\QQ(\x) \geq \QQ(F(\x))$. So we have to show that $\QQ(\x) \geq \QQ(F^L(\x))$ and $\QQ(\x) \geq \QQ(F^R(\x))$. 

We decompose a pseudo-domain of type $L$ or $R$ as the $*$ of a number of punctured rectangles, $4m$-gons that contain $X_1$ and empty rectangles.

Let $\mathbf{u}, \mathbf{v}\in \S(\t H)$ such that there is a punctured rectangle $\a \in A(\mathbf{u} , \mathbf{v})$, then the image of the complementary rectangle $r_{\a}\in Rect(\mathbf{v}, \mathbf{u})$ under $\Pi$ is empty from the points in $Q$. Hence by Equation~\ref{QQ} we have $\QQ(\mathbf{u})=\QQ(\mathbf{v})$.
For $\mathbf{w}, \mathbf{z}\in \S(\t H)$ if there is a $4m$-gons that contains $X_1$ or an empty rectangle from $\mathbf{w}$ to $\mathbf{z}$, from Equation~\ref{QQ} we have $\QQ(\bm w) \geq \QQ(\bm z)$. Hence $\QQ(\x) \geq \QQ(F(\x))$ i.e. $F$ preserves the $Q$-filtration.

Let $C_Q$ (resp. $C'_Q$) be the associated graded object of $C(\t H)$ (resp. $C(\t G)[1]\oplus C(\t G)$). We consider the map induced by $F$ on the filtered objects.
\[F_Q:{C_Q} \longrightarrow {C'_Q} \]

By Lemma~\ref{Homology}, the homology of ${C_Q}$ is carried by the subcomplex ${C_Q}^{(I,\cdots,I)}\oplus{C_Q}^{(J,\cdots,J)}\subset{C_Q}$. So we consider the restriction of $F_Q$ to this subcomplex, and will show that it induces an isomorphism. 

Note that since we consider the induced map on the filtered objects the only pseudo-domains that contribute to $\partial$, are those that do not have intersection with the dots in $Q$.

Let $\x$ be a generator in ${C_Q}^{(I,\cdots,I)}$. Since $F(\x)$ ends up in $(I,\cdots,I)$, the only pseudo-domain that contributes to $F(\x)$ is the trivial domain of Type $L$. Hence $F_Q^L$ restricted to $C_Q^{(I,\cdots,I)}$ is an isomorphism.

The restriction of $F_Q^R$ to ${C_Q}^{(J,\cdots,J)}$, counts pseudo-domains of Type $R$ which are supported in the lift of the row and the column through $O_1$. Therefore we only count $4m$-gons, from a generator of type $(J,\cdots,J)$ to a generator of type $(I,\cdots,I)$.
But for each element in $(J,\cdots,J)$ there is a unique $4m$-gon with this property. Thus $F_Q^R$ restricted to $C_Q^{(J,\cdots,J)}$ is an isomorphism.
\end{proof}

Now we need the following algebraic lemma, for a proof see Theorem 3.2 from \cite{Mc}.

\begin{lemma} \label{nice}
Suppose that $F:C\longrightarrow C'$ is a filtered chain map which induces an isomorphism on the homology of the associated graded object. Then $F$ is a quasi-isomorphism. 
\end{lemma}

\begin{proposition}\label{quasi-final}

The map $F$ is a quasi-isomorphism.
\end{proposition}

\begin{proof}
Use Lemma~\ref{nice} and Proposition~\ref{quasi} for the $\QQ$-filtration to conclude that $F$ is a quasi-isomorphism.  
\end{proof}

This completes the proof of the fact that $H(C)\cong H(C')\cong H(B)\otimes V$, where $V \cong \Z \oplus \Z$ with generators in gradings $0$ and $-1$.

%%%%%%%%%%%%%%%%%%%%%%%%%%%%%%%%%%%%
%%%%%%%%%%%%%%%%%%%%%%%%%%%%%%%%%%%%
%%%%%%%%%%%%%%%%%%%%%%%%%%%%%%%%%%%%
%%%%%%%%%%%%%%%%%%%%%%%%%%%%%%%%%%%%

% ----------------------------------------------------------------
\end{document}